\documentclass[11pt,reqno]{amsart}   	
		
\usepackage{amssymb, latexsym}
\usepackage{amsmath}
\usepackage[toc,page]{appendix}
\usepackage{braket}
\usepackage{breqn}
\usepackage{caption}
\usepackage{color}
\usepackage{dsfont}
\usepackage{esint}
\usepackage[T1]{fontenc}
\usepackage{geometry}         
\usepackage{graphicx}
\usepackage{hyperref}
\usepackage{latexsym}
\usepackage{mathrsfs}
\usepackage{mathtools}
\usepackage{subcaption}
\usepackage{enumerate}
\usepackage{tikz}
\usetikzlibrary{cd}
\usetikzlibrary{positioning}
\usepackage{times}

\usepackage{amssymb, latexsym}
\usepackage{amsmath}
\usepackage[toc,page]{appendix}
\usepackage{braket}
\usepackage{breqn}
\usepackage{caption}
\usepackage{color}
\usepackage{dsfont}
\usepackage{enumitem}
\usepackage{esint}
\usepackage[T1]{fontenc}
\usepackage{geometry}         
\usepackage{graphicx}
\usepackage{hyperref}
\usepackage{latexsym}
\usepackage{mathrsfs}
\usepackage{mathtools}
\usepackage{subcaption}

\theoremstyle{plain}
\numberwithin{equation}{section}
\newtheorem{thm}{Theorem}[section]
\newtheorem{theorem}[thm]{Theorem}
\newtheorem{lemma}[thm]{Lemma}
\newtheorem{corollary}[thm]{Corollary}

\newtheorem{definition}[thm]{Definition}

\newtheorem{proposition}[thm]{Proposition}
\newtheorem{remark}[thm]{Remark}
\newtheorem{assumption}[thm]{Assumption}
\newtheorem{problem}[thm]{Problem}

\def\de{\delta}

\def\ep{\epsilon}
\def\ga{\gamma}
\def\Ga{\Gamma}

\def\la{\lambda}

\def\om{\omega}
\def\Om{\Omega}
\def\pa{\partial}

\def\vphi{\varphi}

\def\N{\mathbb{N}}
\def\P{\mathbb{P}}

\def\R{\mathbb{R}}

\def\Triag{\mathscr{T}}

\def\fa{\forall}
\def\grad{\nabla}
\def\lang{\langle}

\def\rang{\rangle}

\def\Alg{\mathcal{A}}

\def\Lin{\mathcal{L}}

\DeclareMathOperator{\dist}{dist}

\DeclareMathOperator{\dive}{div}

\def\loc{\text{loc}}

\DeclareMathOperator{\supp}{supp}
\DeclareMathOperator{\Sym}{Sym}

\def\til{\widetilde}

\newcommand\beq{\begin{equation}}
\newcommand\eeq{\end{equation}}
\newcommand\bea{\begin{eqnarray}}
\newcommand\eea{\end{eqnarray}}
\newcommand\bi{\begin{itemize}}
\newcommand\ei{\end{itemize}}
\newcommand\ben{\begin{enumerate}}
\newcommand\een{\end{enumerate}}

\usepackage{xcolor}

\everymath{\displaystyle}

\begin{document}

\title[Optimal control of peridynamics]{On the Optimal Control of a Linear Peridynamics Model}

\author{Tadele Mengesha \and Abner J. Salgado \and Joshua M. Siktar}
\address{Department of Mathematics, University of Tennessee, Knoxville TN 37996, USA}

\email[T. Mengesha]{mengesha@utk.edu}
\urladdr[T. Mengesha]{https://sites.google.com/utk.edu/tadelemengesha/}

\email[A.J. Salgado]{asalgad1@utk.edu}
\urladdr[A.J. Salgado]{https://sites.google.com/utk.edu/abnersg/}

\email[J.M. Siktar]{jsiktar@vols.utk.edu}
\urladdr[J.M. Siktar]{https://joshuasiktarcomputationalarchive.weebly.com/}



\date{\today}

\begin{abstract}
    We study a non-local optimal control problem involving a linear, bond-based peridynamics model. In addition to existence and uniqueness of solutions to our problem, we investigate their behavior as the horizon parameter $\de$, which controls the degree of nonlocality, approaches zero. We then study a finite element-based discretization of this problem, its convergence, and the so-called asymptotic compatibility as the discretization parameter $h$ and the horizon parameter $\de$ tend to zero simultaneously.
\end{abstract}

\maketitle


\section{Introduction}\label{introduction}

This paper focuses on an optimal control problem with a system of constraint equations derived from peridynamics (PD), which is a contemporary non-local model in solid mechanics,  \cite{silling2000reformulation, silling2007peridynamic}. PD models do not assume the differentiability (even in the weak sense) of 
pertinent forces acting on a body nor on the resulting displacement vector fields, unlike their local counterparts in continuum mechanics. This feature of PD models makes them attractive to  analyze certain physical phenomena with inherent discontinuities, such as the formation of cracks
 in solids \cite{silling2005peridynamic, silling2010crack, silling2014peridynamic}. In this work, we will focus on the bond-based PD model, where particles in a solid are assumed to expend long distance forces on other particles within a certain radius. With this in mind, we will consider the problem of linearly deforming  a [possibly heterogeneous] elastic solid occupying a domain $\Om \subset \R^n$ to achieve a desired deformation state by applying a certain external force. The deformation field given by $v(x) := x + u(x)$, where $u$ is the displacement, and the external force $g$ are related via the linearized bond-based PD model \cite{du2013analysis, mengesha2014bond,silling2010linearized}
 given by 
\[
    \Lin_{\de} u(x):=\int_{\mathbb{R}^{n}} \mathfrak{f}_{\de} (s[u](x,y), y, x) dy= g(x), \quad x\in \Om, 
\]
where  the vector-valued pairwise force density function $\mathfrak{f}_\de$ along the bond joining material points $x$ and $y$, and the scalar linearized strain field $s[u]$ associated with the displacement $u$ are given by \[
\begin{split}
s[u](x,y)  &= \frac{u(x)-u(y)}{|x - y|}\cdot \frac{x - y}{|x - y|};\\
 \mathfrak{f}_\de (s[u](x,y), y, x) &= H(x, y)k_\de(|x-y|) s[u](x,y) \frac{y - x}{|x - y|}.
\end{split}
\]
In the above, $H(x, y)= \frac{1}{2}(h(x)+h(y))$ serves as material coefficient for some bounded function $h$. The function $k_\de(|x-y|)$ is the interaction kernel that is radial  and  describes the force strength between material points. The parameter $\de > 0$, in the definition of $\Lin_{\de}$, is called the {\em horizon} and measures the degree of non-locality, i.e., the radius within which the interaction forces are considered. We assume that $k_\de(|x-y|) = 0$ if $|x-y|\geq\de$; additional assumptions on the family $\{k_\de\}_{\de>0}$ will be given later.

To quantify the desirability of a displacement state $ u$ subject to the the external force $g$, which will be our control, we introduce an objective functional $I(u, g)$.  This functional will be taken to be  a sum of two parts: one measures, say, the mismatch of the displacement state $u$ and the desired  displacement field, say $u_{des}$, and the other penalizes the control $g$ and serves as a regularizer. 
We will delay the exact form  of the objective functional until the next section, but the optimal control problem of interest of the paper can now be stated as  
\begin{equation}\label{GenericOptCond}
    \begin{dcases}
    \min \{ I(u, g) \ | \ (u,g) \in X_{ad} \times Z_{ad} \}, \\
    \Lin_{\de} u \ = \ g \quad \text{in $\Om$},
    \end{dcases}
\end{equation}
where the  
admissible set $X_{ad} \times Z_{ad}$ will be specified in the next section. As described above, the state equation, codified by the operator $\Lin_{\de}$, will be a strongly coupled linear system of integral equations.  The definition of $\Lin_{\de}u$ requires the knowledge of the state $u$ outside of the domain $\Om$, up to a boundary layer of thickness $\de$.  Thus, we close the state equation in \eqref{GenericOptCond}, by assigning $u$ to be a fixed displacement field $u_0$ in the boundary layer which we call the {\em nonlocal Dirichlet boundary condition.}

In this work, we prove the well-posedness of \eqref{GenericOptCond} for a more general class of objective functionals and a broader class of interaction kernels $k_\de$ that include fractional-type kernels. We also study the behavior of the optimal pair $(\overline{u}, \overline{g})$ as a function of the horizon $\de$. In fact, we demonstrate that in the vanishing horizon limit $\de \to 0^+$ the integral equation-based optimal control problem \eqref{GenericOptCond} converges, in a certain sense,  to a differential-equation-based optimal control problem.  Well-posedness as well as vanishing nonlocality limit for the state equations have been studied in  \cite{mengesha2014bond, scott2018fractional, zhou2010mathematical}.   
In addition, we consider the numerical approximation of solutions to \eqref{GenericOptCond} via the first-order optimality conditions. The discrete problem will involve two parameters: the discretization parameter $h$ and the horizon $\de$. We will  show that we have convergence, not only when $h$ tends to zero, but that we also have \emph{asymptotic compatibility} (see \cite{tian2014asymptotically}), in the sense that the limit is unique regardless of the path we use to let $h \to 0^+$ and $\de \to 0^+$.

 While literature on optimal control problems is immense, we cite some works that are related to the current study. The optimal control problem when the state equation is  a scalar fractional or non-local equation is studied in  
 \cite{antil2017note, antil2020optimal, antil2021optimal, burkovska2021optimization, d2019priori, munoz2022local}. 
 The papers \cite{bonito2018numerical, borthagaray2019weighted, acosta2017fractional, d2019priori, ern2017finite} study the finite element analysis of optimal control  problems of fractional or nonlocal equations. 
For our approach of  using the first-order optimality conditions in order to approximate the continuous problem with the corresponding discrete problems, we refer the reader to \cite{antil2021optimal, casas2015second, casas2012optimality, d2019priori, otarola2019maximum} for more on this subject matter. To the best of our knowledge the optimal control problem for a strongly coupled system of nonlocal equations of peridynamic-type has not been studied in the literature; the current work makes a contribution in that direction.  We also mention that while the present work focuses on the basic linear bond-based peridynamic model, similar analysis can be done on the more general state-based peridynamics \cite{silling2007peridynamic} as well as other nonlinear models, like those studied in \cite{mengesha2015variational}. This and other related issues will be addressed in future work. 

We now outline the contents of the rest of the paper. First, Section \ref{sec:Notation} states the problems to be studied, with all notation made precise. Section \ref{properties} highlights some structural  properties of the function space of interest such as 
compact embedding. The framework from which the well-posedness of our local and non-local optimal control problems can be deduced is carried over in Section \ref{oc}. The remaining sections study the relationship between our problems as $\de$ and $h$ change: Section \ref{gaconv} considers $\Ga$-convergence results as $\de \rightarrow 0^+$; Section \ref{FEMAnalysis} features finite element analyses for the local and non-local problems as $h \rightarrow 0^+$; and Section \ref{compatibility} proves the asymptotic compatibility of limits as $\de$ and $h$ both tend to $0$. 


\section{Problem formulation}\label{sec:Notation}


\subsection{Notation and assumptions}
Let us begin by introducing some notation; first, by $A \lesssim B$ we mean that there is a nonessential constant $c$, such that $A \leq cB$. In addition, $A \sim B$ means $A \lesssim B \lesssim A$. We assume throughout the paper that $\Om \subset \R^n$ is an open, bounded domain with a Lipschitz boundary, and denote $\Om_{\de} := \Om \cup \{x \in \R^n \ | \ \dist(x, \Om) < \de\}$, where $\de > 0$ is the horizon parameter. By volumetric boundary we mean the boundary layer $\Om_{\de}\setminus \Om$  surrounding $\Om$.  For any $r > 0$ and $x_0 \in \mathbb{R}^{n}$, we denote a ball centered at $x_0$ with radius $r$ by $B_r(x_0)$. Next we provide assumptions on our kernels which are adopted from \cite{bourgain2001another, munoz2022local}. 
 
 \begin{assumption}[Kernel assumptions]\label{kernelAssump}
We assume that $\{k_{\de}\}_{\de > 0}$ is a family of radial, non-negative, kernels in $L^1(\R^n)$ supported in $B_\de(0)$  such that 
\begin{equation}\label{kernelNormalization}
\begin{split}
    \int_{\R^n}k_{\de}(|\xi|)d\xi \ &= \ 1,\\
      \lim_{\de \rightarrow 0^+}\int_{\R^n \setminus B_\epsilon(0)}k_{\de}(|\xi|)d\xi \ &=0,\quad \text{for all  $\epsilon >0$.}
    \end{split}
\end{equation}
The above two conditions say that the family of $L^1$ functions $\{k_{\de}\}_{\de > 0}$ converges to the Dirac measure $\de_0$ in the sense of measures.  
We also assume that for each $\de>0,$ $k_{\de}(r)r^{-2}$ is non-increasing in $r$. 
 \end{assumption}
 
Given any $L^1(\mathbb{R}^{n})$-function $k(|\xi|)$ supported on the unit ball $B_1(0)$, the family $k_{\de}(|\xi|) = \de ^{-n}k\bigg({|\xi|\over \de}\bigg)$ satisfies \eqref{kernelNormalization}; for other  nontrivial kernels satisfying the above conditions see \cite{bourgain2001another}. 
To properly define our function spaces and norms, we introduce some additional notation. First, given $u:\Om_{\de} \to \R^n$ measurable, we let $Du$ represent the \textbf{projected difference} defined as 
\[Du(x, y) := (u(x) - u(y)) \cdot \frac{(x - y)}{|x - y|}.\]
This quantity is the trace of $(u(x) - u(y)) \otimes \frac{x - y}{|x - y|}$. Notice then that the linearized strain field $s[u](x, y)$ is given by $s[u](x, y)=\frac{Du(x, y)}{|x - y|}$. Using these notations, the  vector-valued nonlocal operator $\Lin_{\de}$ is given by \[
\Lin_{\de} u(x) = \int_{\Om_{\de}}H(x,y)k_\de(|x-y|) \frac{Du(x, y)}{|x - y|}\frac{y - x}{|x - y|}dy
\]
whenever it makes sense. We notice that for $u, v \in C^{\infty}_c(\Om;\R^n)$, see \cite[Proposition A.5]{foghem2022general},   
\[
\begin{split}
\int_{\Om} \Lin_{\de} u(x) \cdot v(x)dx &=  {1\over 2} \iint_{\mathcal{D}_{\de}}H(x, y)k_{\de}(x - y)\frac{Du(x, y)}{|x - y|}\frac{Dv(x, y)}{|x - y|}dxdy\\
&=:B_\de(u, v),
\end{split}
\]  
where $\mathcal{D}_{\de} = (\Om \times \Om_{\de}) \cup (\Om_{\de} \times \Om).$  The latter defines a bi-linear form and we understand the strongly coupled system of nonlocal equations for the state $u$, $ \Lin_{\de} u = g$,  in the weak sense as the Euler-Lagrange equation for the corresponding quadratic potential energy 
\begin{equation}\label{Energy1st}
 {1\over 2}\iint_{\mathcal{D}_{\de} } H(x, y) k_{\de}(|x-y|) \left|\frac{Du(x, y)}{|x - y|}\right|^{2}  dxdy  -\int_{\Om} {g}(x)\cdot {u}(x)dx
\end{equation}
defined on an appropriate space of functions with a displacement field on the nonlocal boundary. 
Recall that $H(x,y) = {1\over 2}(h(x) + h(y))$ and  that there are positive  constants $h_{min}$ and $h_{max}$
 such  that $h_{min} \leq h(x) \leq h_{max}$ for all $x\in \Om_{\de}$. 
 With  this assumption on $H$ and for $g\in L^2(\Om;\mathbb{R}^{n})$, the energy  in \eqref{Energy1st} is finite for $u: \Om_{\de} \to \R^n$  measurable such that 
 \[
  u|_\Om\in L^2(\Om;\R^n),\quad \text{and }\, \iint_{\mathcal{D}_{\de} } k_{\de}(|x-y|) \left|{Du(x, y) \over  |y-x|}\right|^{2}  dxdy  <  \infty.
 \]
 We denote this space of functions by $X(\Om_{\de}; \R^n)$; i.   e. 
 \[
 X(\Om_{\de}; \R^n) = \left\{  u|_\Om\in L^2(\Om; \R^n) \  \middle| \ \iint_{\mathcal{D}_{\de} } k_{\de}(|x-y|) \left|{Du(x, y) \over  |y-x|}\right|^{2}  dxdy  <  \infty\right\}.
 \]
 We also introduce the corresponding space of functions having a zero nonlocal boundary condition as 
 \[
  X_0(\Om_{\de}; \R^n) = \left\{  u\in X(\Om_{\de}; \R^n)\  \middle| \ u=0 \,\, \text{on $\Om_{\de}  \setminus \Om$}\right\}.
\]
It is not difficult to show that the spaces $ X(\Om_{\de}; \R^n)$ and $ X_0(\Om_{\de}; \R^n)$ are normed spaces with the norm
\begin{equation}\label{X-norm}
\|u\|_{X(\Om_{\de}; \R^n)} \ := \left(  \|u\|_{L^2(\Om; \R^n)}^2 + [u]_{X(\Om_{\de}; \R^n)}^{2}\right)^{{1\over 2}}
\end{equation}
where $[u]^2_{X(\Om_{\de}; \R^n)} = \iint_{\mathcal{D}_{\de} } k_{\de}(|x-y|) \left|{Du(x, y) \over  |y-x|}\right|^{2}  dxdy.$ 
Notice that the form $B_{\de} : X(\Om_{\de}; \R^n) \times X(\Om_{\de}; \R^n) \to \R$ is a well defined continuous bi-linear form. 

One objective of this work is to make connections between the non-local optimal control problem and a local control problem as $\de \rightarrow 0^+$. As we will show, the corresponding bi-linear form of interest is 
\begin{equation}\label{Bh0}
    B_0(u, v) \ := \ \frac{1}{n(n + 2)}\int_{\Om}h(x)(2\lang\Sym(\grad u(x)), \Sym(\grad v(x))\rang_F + \dive(u(x))\dive(v(x)))dx, 
\end{equation}
where $\lang \cdot, \cdot\rang_F$ is the Fröbenius inner product on matrices:
\[
  \lang A, B\rang_F \ := \ \sum^{n}_{i = 1}\sum^{n}_{j = 1}a_{i, j}b_{i, j}, \qquad \forall A,B \in \R^{n \times n}.
\]
It turns out that the appropriate energy space for the resulting local problem is the classical space 
\begin{equation}\label{H^10Def}
    H^1_0(\Om; \R^n) \ := \ \{u \in L^2(\Om; \R^n) \ | \ \nabla u \in L^2(\Om; \R^{n \times n}), \ u \ = \ 0 \ \text{on} \ \pa \Om\},
    \end{equation}
with the natural norm 
\begin{equation}\label{H^1-norm}
    \|u\|_{H^1(\Om; \R^n)} \ := \ \left( \|u\|_{L^2(\Om; \R^n)}^2 + [u]_{H^1(\Om; \R^n)}^2 \right)^{\frac{1}{2}},
\end{equation}
and corresponding semi-norm
\begin{equation}\label{H^1-seminorm}
    [u]_{H^1(\Om; \R^n)} \ := \ \|\grad u\|_{L^2(\Om; \R^{n \times n})}.
\end{equation}

Now, to state the optimal control problem of interest precisely, we define the pertinent objective functional. As we mentioned earlier the functional will be taken to be the sum of two terms. The first is a quality functional $Q:X_{ad}\subset X(\Om_{\de}; \R^n) \to [0, \infty)$, that assigns a certain value $Q(u)$ to each admissible displacement field depending on a certain criteria. For example, given a desired displacement state $u_{des}$, we may want a state $u$ that matches $u_{des}$ as closely as possible. In this case we wish to choose $u$ that keeps the mismatch between $u$ and $u_{des}$ to the minimum. The mismatch may be defined as a weighted squared error $\int_{\Om} \ga(x)|u(x) - u_{des}(x)|^{2} dx$ for some $0 \leq \ga\in L^\infty(\Om)$. Notice that by choosing $\ga$ appropriately, we may seek to match the desired state only on a portion of the domain.   
More generally, we would want the quality functional to have the form 
\[
Q(u) = \int_{\Om}F(x, u(x))dx,
\]
where the integrand $F: \Om \times \R^n \to \R$ possesses the following properties:
\begin{enumerate}
\item For all $v \in \R^n$ the mapping $x \mapsto F(x,v)$ is measurable;
\item For all $x \in \Om$ the mapping $v \mapsto F(x,v)$ is continuous and convex;
\item There exist constant $c_1>0$ and $l \in L^1(\Om) $ for which
\begin{equation}\label{Gqcoercive}
    |F(x, v)| \ \leq \ c_1|v|^2 + l(x)
\end{equation}
for all $x \in \Om$ and all $v \in \R^n$.
\end{enumerate}

The second part of the objective functional is a cost functional associated with the external force. We seek a forcing term $g$ whose associated displacement has the desired quality while keeping the cost as minimal as possible. Typically, we take this cost functional, $C(g)$, to be a weighted $L^2$-norm of $g$ of the form 
\[
C(g) = \int_{\Om} \Ga(x) |g(x)|^{2} dx, 
\]
for some $0 < \Ga \in L^1(\Om)$.
 To that end, we take the admissible control space to be $Z_{\text{ad}}$, a nonempty, closed, convex, and bounded subset of $L^2(\Om; \R^n)$, and it takes the form
\begin{equation}\label{linearAdSet}
Z_{\text{ad}} = \{z \in L^2(\Om;\R^n)  \ | \ a(x) \preceq z \preceq b(x)\}
\end{equation}
for some $a, b\in L^{\infty}(\Om; \mathbb{R}^n)$, where $a\preceq b$ means $[a]_i \leq  [b]_i$ for all $i \in \{1, 2, \dots, n\}$. Without loss of generality, we shall assume that $0 \in Z_{\text{ad}}$.

In summary, the objective function we will be working with is of the form
\begin{equation}\label{dmethod2CostIntro}
    I(u, g) \ := \ \int_{\Om}F(x, u(x))dx + \int_{\Om} \Ga(x) |g(x)|^{2} dx
\end{equation}
under the above assumptions on $F$ and $\Ga$. 


\subsection{Problem set up}
Now that we have specified the different function spaces as well as bi-linear forms of interest, we are now ready to  precisely pose the optimal control problems.  
The first one is the optimal control problem of the coupled system of nonlocal equations.  Given a boundary data $u_0\in X(\Om_{\de}; \R^n)$, the problem is
finding a pair $(\overline{u_{\de}}, \overline{g_{\de}})\in X(\Om_{\de}; \R^n) \times Z_{\text{ad}}$ such that 
\begin{equation}\label{stateMin-Nonzero}
    I(\overline{u_{\de}}, \overline{g_{\de}}) = \min I(u_{\de}, g_{\de}),
\end{equation} 
where the minimization is over pairs $(u_{\de}, g_{\de}) \in X(\Om_{\de}; \R^n) \times Z_{\text{ad}}$ that satisfy 
\begin{equation}\label{stateSystem-Nonzero}
u_{\de} - u_0 \in X_0(\Om_{\de}; \R^n),\,\,\text{and} \,\, B_{\de}(u_{\de}, v) \ = \ \lang g_{\de}, v\rang, \, \text{for all }\, v \in X_0(\Om_{\de}; \R^n).
\end{equation}
Here we use the notation $\langle\cdot,\cdot\rangle$ for the $L^2$-inner product. 
We remark that without loss of generality we may assume that $u_0 = 0$ in the above formulation. Indeed, if $u_{\de}$ solves  \eqref{stateSystem-Nonzero} and we set $e_{\de} := u_{\de} -u_0$, then  $e_{\de}\in X_0(\Om_{\de}; \R^n)$ and  
\begin{equation}\label{zeroBdryCond}
B_{\de}(e_{\de}, v) \ = \ \lang g_{\de}, v\rang + B_\de(u_0, v), \,\,\text{for all $v\in X_0(\Om_{\de}; \R^n)$}.
\end{equation}
After noting that the map $v\to B_\de(u_0, v)$ is a bounded linear functional on $X_0(\Om_{\de}; \R^n)$, the right hand side of \eqref{zeroBdryCond} can  be viewed to define a duality pairing between $X_0(\Om_{\de}; \R^n)$ and its dual.  The objective functional as a function of $(e_\de, g_\de)$ with still have the form as \eqref{dmethod2CostIntro} with an integrand $\tilde{F}(x, e) = F(x, e + u_0(x))$. Notice that $\tilde{F}$ the exact same properties as $F$.  

With this simplification at hand, we summarize the problem as follows.  
\begin{problem}[Non-local continuous problem]\label{nlctsProb}
Find a pair $(\overline{u_{\de}}, \overline{g_{\de}})\in X_0(\Om_{\de}; \R^n) \times Z_{\text{ad}}$ such that 
\begin{equation}\label{stateMin}
    I(\overline{u_{\de}}, \overline{g_{\de}}) = \min I(u_{\de}, g_{\de}) 
\end{equation} 
where the minimization is over pairs $(u_{\de}, g_{\de}) \in X_0(\Om_{\de}; \R^n) \times Z_{\text{ad}}$ that satisfy 
\begin{equation}\label{stateSystem}
 B_{\de}(u_{\de}, v) \ = \ \lang g_{\de}, v\rang, \, \text{for all }\, v \in X_0(\Om_{\de}; \R^n).
\end{equation}
\end{problem}
The effective admissible class of pairs for this nonlocal optimal control problem is 
\begin{equation}\label{nonlocalAdmiClass}
\Alg^{\de}=\{  (w, f) \in X_0(\Om_{\de}; \R^n) \times Z_{\text{ad}} | B_{\de}(w, v) \ = \ \lang f, v\rang, \, \text{for all }\, v \in X_0(\Om_{\de}; \R^n\}.
\end{equation}
We are also interested in the behavior of the above nonlocal optimal control problem in the limit of vanishing nonlocality as quantified by $\de$ which turns out to be a local problem.

\begin{problem}[Local continuous problem]\label{lCtsProb}
Find a pair $(\overline{u}, \overline{g}) \in H^1_0(\Om; \R^n) \times Z_{\text{ad}}$ such that
\begin{equation}\label{locObjFunc}
    I(\overline{u}, \overline{g}) \ = \ \min I(u, g),
\end{equation}
where the minimization is over pairs $(u, g) \in H^1_0(\Om; \R^n) \times Z_{\text{ad}}$ that satisfy
\begin{equation}\label{locProb}
    B_0(u, v) \ = \ \lang g, v\rang, \ \quad \fa v \in H^1_0(\Om; \R^n).
\end{equation}
\end{problem}
As before, the effective admissible class of pairs for the control problem is 
\begin{equation}\label{admissLoc}
    \mathcal{A}^{\loc} \ := \ \{(w, f) \in H^1_0(\Om; \R^n) \times Z_{\text{ad}} |   B_0(w, v) \ = \ \lang f, v\rang \quad \fa \ v \in H^1_0(\Om; \R^n) \}.
\end{equation}
We now introduce notation for our finite element scheme and discretized problems. The family of meshes $\{\Triag_h\}_{h > 0}$ discretizing $\Om_{\de}$ is assumed to be quasi-uniform and of size $h$. Let $X_h \subset X_0(\Om_{\de}; \R^n)$ denote the space of continuous, piecewise linear, functions subject to the mesh with zero non-local boundary data, i.e.,
\begin{equation}\label{XhDef}
    X_h \ := \ \{w_h \in C^0(\overline{\Om_{\de}}; \R^n) \ | \ w_h|_T \in \mathcal{P}_1(T; \R^n) \ \fa T \in \Triag_h, w_h \ = \ 0 \ \text{on} \ \Om_{\de} \setminus \Om\},
\end{equation}
and $X_{\de, h}$ will denote this same function space, albeit with a different norm. For the local discrete problem, this space will denoted as $X_h$ and equipped with the norm \eqref{H^1-norm}, and for the non-local discrete problem, this space will instead be denoted as $X_{\de, h}$ and equipped with the norm \eqref{X-norm}. Similarly, let $Z_h$ denote the piecewise constant functions with respect to our mesh, i.e.,
\begin{equation}\label{ZhDef}
    Z_h \ := \ \{z_h \in L^\infty(\Om_{\de}; \R^N) \ | \ z_h|_T \in \mathcal{P}_0(T; \R^n) \ \fa T \in \Triag_h\}.
\end{equation}
Here and henceforth, we denote the space of vector-valued polynomials of degree $m$ as
\begin{equation}\label{P^mSpace}
    \mathcal{P}_m(T; \R^n) \ := \ \left\{ \sum_{ \alpha \in \mathbb{N}_0^n\ :\  \sum_{i=1}^n \alpha_i \leq m } v_\alpha x_1^{\alpha_1} \cdots x_n^{\alpha_n}  \ \middle| \ v_\alpha \in \R^n, \ x = (x_1, \ldots, x_n)^\intercal \in T \right\}.
\end{equation}
We will use $X_{\de, h}$ and $X_h$, as appropriate, to discretize the state space, and $Z_h$ to discretize the control space. Now we may state our non-local and local discrete problems.
\begin{problem}[Non-local discrete problem]\label{nlDiscProb}
Find a pair $(\overline{u_{\de, h}}, \overline{g_{\de, h}})\in X_{\de, h} \times Z_h$ such that 
\begin{equation}\label{stateMinNLD}
    I(\overline{u_{\de, h}}, \overline{g_{\de, h}}) = \min I(u_{\de, h}, \ g_{\de, h})
\end{equation} 
where the minimization is over pairs $(u_{\de, h}, g_{\de, h}) \in X_{\de, h} \times Z_h$ that satisfy
\begin{equation}\label{stateSystemNLD}
B_{\de}(u_{\de, h}, v_{\de, h}) \ = \ \lang g_{\de, h}, v_{\de, h}\rang \ \quad \fa v_{\de, h} \in X_{\de, h}.
\end{equation}
\end{problem}
The effective admissible class of pairs for the above nonlocal discrete problem is 
\begin{equation}\label{nonlocalAdmiClassNLD}
    \Alg^{\de}_h \ := \ \{(w_{\de,h}, f_{\de,h}) \in X_{\de, h} \times Z_h, w_{\de, h} \ | \ B_{\de}(w_{\de, h}, v_{\de, h}) \ = \ \lang f_{\de, h}, v_{\de, h}\rang \ \quad \fa v_{\de, h} \in X_{\de, h}\,\}    
\end{equation}
Finally we state the local discrete optimal control problem. 
\begin{problem}[Local discrete problem]\label{lDiscProb}
Find a pair $(\overline{u_h}, \overline{g_h})\in X_h \times Z_h$ such that 
\begin{equation}\label{stateMinLocDiscrete}
    I(\overline{u_h}, \overline{g_h}) \ = \ \min I(u_h, g_h),
\end{equation} 
where the minimization is over pairs $(u_h, g_h) \in X_h \times Z_h$ that satisfy
\begin{equation}\label{stateSystemLocDiscrete}
B_0(u_h, v_h) \ = \ \lang g_h, v_h\rang \ \quad \fa v_h \in X_h.
\end{equation}

\end{problem}
The effective admissible class of pairs for this local problem is 
\begin{equation}\label{nonlocalAdmiClassLocDiscrete}
    \Alg^{\text{loc}}_h \ := \ \{(w_h, f_h) \in X_h \times Z_h \ | \ B_0(w_h, v_h) \ = \ \lang f_h, v_h\rang, \ \quad \fa v_h \in X_h\}.
\end{equation}

Note that in each problem, the state equation governs the relationship between the force [control] and the displacement [state] that must take place in any admissible solution.

\section{Properties of function spaces}\label{properties}
In this section we state and prove some structural properties of the function spaces 
$X(\Om_{\de}; \R^n)$ and $X_0(\Om_{\de}; \R^n)$ defined in the previous section. We begin noting that the function spaces  
are separable Hilbert spaces with the following inner product defined for $u, v\in X(\Om_{\de}; \R^n)$:
\[
[[u, v]]:= \langle u, v\rangle +[u,v]_{X},
\]
where $ [u,v]_{X}=\iint_{\mathcal{D}_{\de}} k_{\de}(|x-y|) {Du(x,y)\over |x-y|}{Dv(x,y)\over |x-y|}dydx. $ It is obvious that $ [u,u]_{X} = [u]_{X}^{2}$, and that, under the working assumption on $H$, we have that ${h_{min}\over 2} [u]_{X}^{2} \leq B_{\de}(u,u) \leq {h_{max}\over 2} [u]^{2}_{X}.$ 
Moreover, for $u \in X_0(\Om_{\de}; \R^n)$, we have 
\[ 
[u]_{X(\Om_{\de}; \R^n)}^{2} = \iint_{\Om_{\de}\times \Om_{\de}} k_{\de}(|x-y|) \left|{Du(x,y)\over |x-y|}\right|^{2}
dydx\] 
which we also use as a seminorm. It then follows from \cite{mengesha2020fractional} that if $\til{u}$ is the zero extension of $u$ to $\R^n$ then there exists a constant $C = C(\de, p) > 0$ such that, for any open set $B$ containing $\Om_{\de}$, we have
\begin{equation}\label{ourExtensionIneq}
    [\til{u}]_{X(B; \R^n)} \ \leq \ C\|u\|_{X(\Om_{\de}; \R^n)}.
\end{equation}
In particular, the constant is independent of $B$, and we may select $B := \R^n$, where we define
\[
  X(\mathbb{R}^{n}; \mathbb{R}^{n}) = \left\{u\in L^2(\R^n) \middle| \iint_{\mathbb{R}^{2n}}k_{\de}(|x-y|) {Du(x,y)\over |x-y|} dydx < \infty \right\}.
\]
We now seek to demonstrate a continuous embedding result for Sobolev spaces into the space $X(\Om_{\de};\R^n)$. To accomplish this, we need a quantitative version of continuity in the $L^2$-norm; a local, scalar-valued analogue is discussed and proven in \cite{bourgain2001another}.

\begin{lemma}[Quantitative $L^2$-continuity]\label{xiSquaredTermSym}
For any $\xi \in \R^n\setminus\{0\}$, and all $v \in H^1(\R^n; \R^n)$ we have
\begin{equation}\label{xiSquaredTermEqSym}
    \int_{\R^n}\left|(v(y + \xi) - v(y)) \cdot \frac{\xi}{|\xi|}\right|^2dy \ \leq \ |\xi|^2\|\Sym(\grad v)\|^2_{L^2(\R^n; \R^{n \times n})}.
\end{equation}
\end{lemma}
\begin{proof}
We first prove the desired claim in the special case where $v \in C^{\infty}(\R^n; \R^n)$. Fix $\xi \in \R^n\setminus\{0\}$. Then by the Chain Rule, the Mean-Value Theorem for integrals, and the Cauchy-Schwarz Inequality, we have
\begin{multline}\label{xiSquaredTermSymEq1}
     \int_{\R^n}\left|(v(y + \xi) - v(y)) \cdot \frac{\xi}{|\xi|}\right|^2dy \ = \ 
     \frac{1}{|\xi|^2}\int_{\R^n}\left|\int^{1}_{0}\grad v(y + t\xi)\xi \cdot \xi dt\right|^2dy \\
      \ \leq \ \frac{1}{|\xi|^2}\int_{\R^n}\int^{1}_{0}|\Sym(\grad v(y + t\xi))\xi \cdot \xi|^2dtdy \ \leq \  |\xi|^2\|\Sym(\grad v)\|^2_{L^2(\R^n; \R^{n \times n})},
\end{multline}
where in the last step we have used invariance of the $L^2$-norm under translations, demonstrating the inequality for  
 $v \in C^{\infty}(\R^n; \R^n)$. The general case for $v \in H^1(\R^n; \R^n)$ follows  by density. 
\end{proof}

The estimate of Lemma~\ref{xiSquaredTermSym} will now be used to prove a continuous embedding result.

\begin{lemma}[Continuous embedding]\label{ctsXW}
For all $\de > 0$, we have
\begin{equation}\label{ctsXWEq}
\begin{split}
    [v]_{X(\Om_{\de}; \R^n)} \    & \lesssim \ |v|_{H^1(\Om; \R^n)}, \qquad \forall v \in H^1_0(\Om; \R^n).
    \end{split}
\end{equation}
That is, $H^1_0(\Om; \R^n) \hookrightarrow X_0(\Om_{\de}; \R^n)$, and the constant is independent of $\de$.
\end{lemma}
\begin{proof}
Since $\pa \Om$ is Lipschitz, for any $v \in H^1_0(\Om; \R^n)$ its extension by zero outside of  $\Om$ is in $H^1_0(\mathbb{R}^{n}, \R^n)$ vanishing almost everywhere outside of $\Om.$ 
Now for any $\de>0$, we have 
\begin{equation}\label{ctsXWEq1}
\begin{split}
    [v]^2_{X(\Om_{\de}; \R^n)} &= \iint_{\Om_{\de}\times \Om_{\de}} k_{\de}(|x-y|) \left|{Du(x,y)\over |x-y|}\right|^{2}
dydx\\
&\leq \int_{B(0, \de)}\frac{k_{\de}(\xi)}{|\xi|^2}\int_{\R^n}\left|(v(y + \xi) - v(y)) \cdot \frac{\xi}{|\xi|}\right|^2dyd\xi,
\end{split}
\end{equation}
where we have used that $\supp(k_{\de}) \subset B(0, \de)$. Now our expression is in a form on which we can use Lemma \ref{xiSquaredTermSym} on the inner integral to conclude that
\begin{equation}\label{ctsXWEq2}
   \int_{B(0, \de)}\frac{k_{\de}(\xi)}{|\xi|^2}\int_{\R^n}\left|(v(y + \xi) - v(y)) \cdot \frac{\xi}{|\xi|}\right|^2dyd\xi \ \leq \ \|\Sym(\grad v)\|^2_{L^2(\R^n; \R^{n \times n})} \leq \| \grad v \|_{L^2(\Om; \R^{n \times n})}^2,
\end{equation}
which completes the proof.
\end{proof}

Next we show that compactly supported smooth functions are dense in $X_0(\Om_{\de}; \R^n)$. 
\begin{lemma}[Density]\label{bddSupVecFieldDense}
The set 
\[
  \left\{ v \in X(\R^n; \R^n) \ \middle| \  \exists R>0: \supp(v) \subset B(0,R) \right\}
\]
is dense in $X(\R^n; \R^n)$.
\end{lemma}
\begin{proof}
Let $\vphi \in C_0^{0,1}(\R^n)$ be such that $0 \leq \vphi \leq 1$, $\vphi \equiv 1$ in $B(0,1)$, $\supp (\vphi) \subset B(0,2)$, and $\|\grad \vphi \|_{L^\infty(\R^n;\R^n)} \leq 1$. For $R>0$ define $\vphi_R(x) = \vphi(x/R)$ and $\psi_R := 1 - \vphi_R$.

Let $u \in X(\R^n; \R^n)$, and we claim that 
\begin{equation}\label{EjDensTarget}
    \lim_{R \rightarrow \infty}[u - u\vphi_R]_{X(\Om_{\de}; \R^n)}^{2}  \ = \ 0.
\end{equation}
To this end, we compute
\begin{equation}\label{bddSuppVecFieldDenseEq2}
\begin{aligned} 
    [u\psi_R]^{2}_{X(\Om_{\de}; \R^n)} &\ = \iint_{\R^{2n}}|D(u\psi_R)(x, y)|^2\frac{k_{\de}(x - y)}{|x - y|^2}dxdy \\
    \ &\lesssim \ \iint_{\R^{2n}}|\psi_R(x) - \psi_R(y)|^2 \left|u(x) \cdot \frac{x - y}{|x - y|}\right|^2\frac{k_{\de}(x - y)}{|x - y|^2}dxdy \\ & + \iint_{\R^{2n}}\psi_R(y)^2\left|(u(x) - u(y)) \cdot \frac{x - y}{|x - y|}\right|^2\frac{k_{\de}(x - y)}{|x - y|^2}dxdy.
\end{aligned}
\end{equation}
By the Dominated Convergence Theorem we deduce
\begin{equation}\label{bddSuppVecFieldDenseEq3}
    \lim_{R \rightarrow \infty} \iint_{\R^{2n}}\psi_R(y)^2\left|(u(x) - u(y)) \cdot \frac{x - y}{|x - y|}\right|^2\frac{k_{\de}(x - y)}{|x - y|^2}dxdy \ = \ 0.
\end{equation}
Now, to handle the first integral in \eqref{bddSuppVecFieldDenseEq2} we define
\begin{equation}\label{bddSuppVecFieldDenseEq4}
    K_R(x) \ := \ \int_{\R^n}|\psi_R(x) - \psi_R(y)|^2\frac{k_{\de}(x - y)}{|x - y|^2}dy.
\end{equation}
By using the conditions on $k_{\de}$ and the Lipschitz continuity of $\psi$, the sequence $|K_R(x)|$ is uniformly bounded in $R$ and in $x$. Further, $K_R(x) \rightarrow 0$ pointwise on $\R^n$ as $R \rightarrow \infty$, so we may, once again,  use the Dominated Convergence Theorem to conclude that
\begin{equation}\label{bddSuppVecFieldDenseEq6}
    \lim_{R \rightarrow \infty}\int_{\R^n}|u(x)|^2K_R(x)dx \ = \ 0,
\end{equation}
proving \eqref{EjDensTarget}. Finally, with one more application of the Dominated Convergence Theorem, we see that
\begin{equation}\label{bddSuppVecFieldDenseEq7}
    \lim_{R \rightarrow \infty}\|u - u\vphi_R\|^2_{L^2(\R^n; \R^n)} \ = \ \lim_{R \rightarrow \infty}\|u\psi_R\|^2_{L^2(\R^n; \R^n)} \ = \ 0,
\end{equation}
and this lets us complete the proof, with $\{u\vphi_R\}^{\infty}_{R = 1}$ as our approximating family of functions in $X(\R^n; \R^n)$ with bounded support.
\end{proof}

The following result is analogous to \cite[Proposition 4.1]{jarohsstrong} and \cite[Lemma 5.2]{scott2018fractional}.

\begin{lemma}[Mollification]\label{cptSupDjDense}
Let $u \in X(\R^n; \R^n)$ be a vector field that vanishes outside a compact subset of $\R^n$. For $\ep>0$ denote by $\eta_\ep$ a standard mollifier, and $u_\ep = u * \eta_\ep$. Then, for $\ep$ sufficiently small, we have
$u_{\ep} \in X(\R^{n}; \R^n)$. Moreover,
\begin{equation}\label{cptSupDjDenseEq}
    \lim_{\ep \rightarrow 0^+}[u - u_{\ep}]^2_{X(\Om_{\de}; \R^n)} \ = \ 0.
\end{equation}
\end{lemma}
\begin{proof}
Let $K \subset \R^{n}$ be a compact set so that $\supp(u) \subset K$. Then 
$u_{\ep} \in C^{\infty}_0(\R^n; \R^n)$ is supported in $K_{\ep} := \{x \in \R^n, \ \dist(x, K) \leq \ep\}$ for any $\ep > 0$. Since the mollifier $\eta_{\ep}$ is even, we may use Hölder's Inequality, Jensen's Inequality, and the identity
\begin{equation}\label{cptSupDjDenseEq0A}
    \int_{\R^n}(\eta_{\ep}*\eta_{\ep})(z)dz \ = \ \left(\int_{\R^n}\eta_{\ep}(z)dz\right)^2 \ = \ 1
\end{equation}
to obtain the estimate
\begin{equation}\label{cptSupDjDenseEq1}
\begin{aligned}
   [u_{\ep}]^{2}_{X(\Om_{\de}; \R^n)} \ &=  \iiiint_{\R^{4n}}\eta_{\ep}(z)\eta_{\ep}(z')Du(x - z, y - z)Du(x - z', y - z')\frac{k_{\de}(x - y)}{|x - y|^2}dz'dzdxdy \ \\
   &= \ 
   \iiiint_{\R^{4n}}\eta_{\ep}(z)\eta_{\ep}(z')Du(x + z' - z, y + z' - z)Du(x, y)\frac{k_{\de}(x - y)}{|x - y|^2}dz'dzdxdy  \ \\
    &=\iiiint_{\R^{4n}}\eta_{\ep}(z - z')\eta_{\ep}(z')Du(x + z, y + z)Du(x, y)\frac{k_{\de}(x - y)}{|x - y|^2}dz'dzdxdy \  \ \\
    &=\iint_{\R^{2n}}\frac{k_{\de}(x - y)}{|x - y|^2}\int_{\R^n}(\eta_{\ep} * \eta_{\ep})(z)Du(x, y)Du(x + z, y + z)dzdxdy  \ \\
    &\leq [u]_{X(\Om_{\de}; \R^n)}\left(\iint_{\R^{2n}}\left(\int_{\R^n}(\eta_{\ep} * \eta_{\ep})(z)|Du(x + z, y + z)|dz\right)^2 \frac{k_{\de}(x - y)}{|x - y|^2}dxdy\right)^{\frac{1}{2}}  \ \\
    &\leq [u]_{X(\Om_{\de}; \R^n)}\left(\int_{\R^n}(\eta_{\ep} * \eta_{\ep})(z)\iint_{\R^{2n}}|Du(x + z, y + z)|^2\frac{k_{\de}(x - y)}{|x - y|^2}dxdydz\right)^{\frac{1}{2}} \ \ \\
    &=[u]^{2}_{X(\Om_{\de}; \R^n)}
\end{aligned}
\end{equation}
which holds for all $\ep > 0$. As a consequence, $u_{\ep} \in X(\R^n; \R^n)$ for all $\ep > 0$.
To proceed further, we define the maps $U, U_{\ep}: \R^n \times \R^n \rightarrow \R$ as
\begin{equation}\label{UUepMaps}
        U(x, y) \ = \ Du(x, y)\sqrt{\frac{k_{\de}(x - y)}{|x - y|^2}} \qquad
        U_{\ep}(x, y) \ = \ Du_{\ep}(x, y)\sqrt{\frac{k_{\de}(x - y)}{|x - y|^2}},
\end{equation}
and these definitions in turn imply that
\begin{equation}\label{cptSupDjDenseEq2}
    [u - u_{\ep}]^2_{X(\Om_{\de}; \R^n)} \ = \ \iint_{\R^{2n}}|D(u - u_{\ep})(x, y)|^2\frac{k_{\de}(x - y)}{|x - y|^2}dxdy \ = \ \|U - U_{\ep}\|^2_{L^2(\R^n \times \R^n)}.
\end{equation}
The proof will be complete once we show that $U_{\ep} \rightarrow U$ in $L^2(\R^n \times \R^n)$ as $\ep \rightarrow 0^+$. As is standard for mollifiers, $u_{\ep} \rightarrow u$ strongly in $L^2(\R^n; \R^n)$, and a.e. pointwise in $\R^n$, both as $\ep \rightarrow 0^+$. Thus by Fatou's Lemma, we get the convergence \begin{equation}\label{cptSupDjDenseEq3}
    \lim_{\ep \rightarrow 0^+}\|U_{\ep}\|_{L^2(\R^n \times \R^n)} \ \geq \ \|U\|_{L^2(\R^n \times \R^n)},
\end{equation}
while the reverse inequality follows from sending $\ep \rightarrow 0^+$ in \eqref{cptSupDjDenseEq1}. This combined with showing $U_{\ep} \rightharpoonup U$ in $L^2(\R^n \times \R^n)$ is enough to show the strong convergence in $L^2(\R^n \times \R^n)$ that we seek, so we focus on proving this weak convergence. Let $V \in L^2(\R^n \times \R^n)$ be arbitrary and define the function
\begin{equation}\label{cptSupDjDenseEq4}
    V_j(x, y) \ := \ \begin{dcases} 
    V(x, y), |x|, |y| \leq j, |x - y| \geq \frac{1}{j} \\
    0, \ \ \ \ \ \ \ \ \ \ \text{otherwise}.
    \end{dcases}
\end{equation}
With this definition in mind, the Dominated Convergence Theorem tells us that $V_j \rightarrow V$ in $L^2(\R^n \times \R^n)$ as $j \rightarrow \infty$. We define $h_1^j, h_2^j: \R^n \rightarrow \R^n$ such that
\begin{equation}\label{cptSupDjDenseEq5}
    \begin{aligned}
        h_1^j(x) \ &:= \ \int_{\R^n}\sqrt{\frac{k_{\de}(x - y)}{|x - y|^2}}V_j(x, y)\frac{x - y}{|x - y|}dy; \\
        h_2^j(y) \ &:= \ \int_{\R^n}\sqrt{\frac{k_{\de}(x - y)}{|x - y|^2}}V_j(x, y)\frac{x - y}{|x - y|}dx. 
    \end{aligned}
\end{equation}
Since $\sqrt{\frac{k_{\de}(x - y)}{|x - y|^2}} \leq 1 + \frac{k_{\de}(x - y)}{|x - y|^2}$ for all $x, y \in \R^n$, we can see that these functions have bounded support, and thus belong to $L^2(\R^n; \R^n)$ for all $j \in \N^+$. Then due to the a.e. convergence $u_{\ep} \rightarrow u$, we have 
\begin{equation}\label{cptSupDjDenseEq6}
    \begin{aligned}
        \lim_{\ep \rightarrow 0^+}\iint_{\R^{2n}}U_{\ep}(x, y)V_j(x, y)dxdy \ &= \ \lim_{\ep \rightarrow 0^+}\iint_{\R^{2n}}[u_{\ep}(x) - u_{\ep}(y)]\sqrt{\frac{k_{\de}(x - y)}{|x - y|^2}}\cdot V_j(x, y)dxdy \ \\
        \ &= \ \lim_{\ep \rightarrow 0^+}\left(\int_{\R^n}u_{\ep}(x)\cdot h_1^j(x)dx - \int_{\R^n}u_{\ep}(y)\cdot h_2^j(y)dy\right) \ \\
        \ &= \ \int_{\R^n}u(x) \cdot h^j_1(x)dx - \int_{\R^n}u(y) \cdot h_2^j(y)dy \\
        \ &= \ \iint_{\R^{2n}}Du(x, y)\sqrt{\frac{k_{\de}(x - y)}{|x - y|^2}}V_j(x, y)dxdy \\
        \ &= \ \iint_{\R^{2n}}U(x, y)V_j(x, y)dxdy,
    \end{aligned}
\end{equation}
which holds for all $j \in \N^+$. Taking a limit supremum in $\ep$, the convergence in \eqref{cptSupDjDenseEq6}, and applying Hölder inequality gives
\begin{equation}\label{cptSupDjDenseEq6A}
    \begin{aligned}
    &\limsup_{\ep \rightarrow 0^+}\left|\iint_{\R^{2n}}(U_{\ep} - U)(x, y)V(x, y)dxdy\right|  \\ 
    \ &= \ \limsup_{\ep \rightarrow 0^+}\left|\iint_{\R^{2n}}(U_{\ep} - U)(x, y)(V - V_j)(x, y)dxdy\right|  \ \\
    &\leq \ \limsup_{\ep \rightarrow 0^+}\|U_{\ep} - U\|_{L^2(\R^n \times \R^n)}\|V - V_j\|_{L^2(\R^n \times \R^n)} \\
    &\leq \ 2\|U\|_{L^2(\R^n \times \R^n)}\|V - V_j\|_{L^2(\R^n \times \R^n)} \ ,
    \end{aligned}
\end{equation}
which holds for all $j \in \N^+$. Finally, due to $V_j \rightarrow V$ in $L^2(\R^n \times \R^n)$, we obtain the limit
\begin{equation}\label{cptSupDjDenseEq7}
    \lim_{\ep \rightarrow 0^+}\iint_{\R^{2n}}(U_{\ep} - U)(x, y)V(x, y)dxdy \ = \ 0,
\end{equation}
and from this it follows that $U_{\ep} \rightharpoonup U$ in $L^2(\R^n \times \R^n)$, completing the proof.
\end{proof}
We can now combine Lemma \ref{bddSupVecFieldDense} and Lemma \ref{cptSupDjDense} to immediately obtain the density of  $C^{\infty}_0(\R^n; \R^n)$ in $X(\R^n; \R^n)$, which we state below as a corollary, see \cite[Remark 4.2]{Jarlocalcompactness} and \cite{foghem2022general} for the scalar case. 
\begin{corollary}[Density]\label{DensityGen}
The space $C^{\infty}_0(\R^n; \R^n)$ is dense in $X(\R^n; \R^n)$.
\end{corollary}

For well-posedness of the state system (specifically, for stability) we shall need a nonlocal Poincar\'e-type inequality. In addition, to understand the behavior of our system in the limit as $\de \rightarrow 0^+$, it is essential that the constant in this inequality is independent of $\de$. The following result was proven in \cite{mengesha2014bond}, but various versions of this inequality are proved in  \cite{bellido2014existence, bellido2015hyperelasticity, buczkowski2022sensitivity, Du1, du2013analysis, foss2019nonlocal, mengesha2014nonlocal, Pon, parini2020compactness, ern2017finite}. 

\begin{proposition}[Nonlocal Poincaré]\label{sharperPoincare}
There exists a $\de_0 > 0$ and a constant $C(\de_0) > 0$ such that for all $\de \in (0, \de_0]$ and $u \in X_0(\Om_{\de}; \R^n)$, we have
\begin{equation}\label{sharperPoincareIneq}
    \|u\|^2_{L^2(\Om; \R^n)} \ \leq \ C(\de_0)\int_{\Om_{\de}}\int_{\Om_{\de}}k_{\de}(x - y) \frac{|Du(x, y)|^2}{|x - y|^2}dxdy.
\end{equation}
\end{proposition}

With the aid of above Poincar\'e-type inequality we may apply Lax-Milgram to deduce the unique solvability of the state equations of the nonlocal optimal control problem stated in the previous section. We summarize this with the following corollary.

\begin{corollary}[Well-posedness of state equation]\label{Well-state}
The state equations \eqref{stateSystem}, \eqref{locProb}, \eqref{stateSystemNLD}, and \eqref{stateSystemLocDiscrete} are uniquely solvable in their corresponding energy spaces.  
\end{corollary}

From standard linear theory, we know that the solution operator of the state equations is linear and continuous.  One important fact we need to demonstrate the solvability of optimal control   problems is the compactness of this solution operator. While for the discrete problems this question is trivial, for the continuous problems it needs a resolution. The compactness of the solution operator is related to the compactness of the image space which, for \eqref{stateSystem}, is  $X_0(\Om_{\de}; \R^n)$; whereas for \eqref{locProb} is $H_{0}^1(\Om;\R^n)$. The compactness of the latter in $L^2(\Om;\R^n)$ is standard.

Below we build a framework needed to ultimately prove the compact embedding os $X_0(\Om_{\de}; \R^n)$ into $L^2(\Om_{\de}; \R^n)$. This is  much akin to the compact embedding results for fractional Sobolev spaces; see, for instance, \cite{demengel2012functional, Di}.  This will largely be based on the results of \cite{Jarlocalcompactness}, see also \cite{gounoue20202}, which we extend to vector-valued functions using a weaker norm that only involves a projected difference quotient. To this end we introduce a definition.

\begin{definition}[Local compactness]\label{localCompact}
If $E$ is a normed vector space, we call a continuous linear operator $T: E \rightarrow L^2(\R^n; \R^n)$ \textbf{locally compact} if the operator $R_KT: E \rightarrow L^2(\R^n; \R^n)$ defined via the truncation function $R_Ku := \mathds{1}_K u$ is a compact operator for every compact subset $K \subset \R^n$.
\end{definition}

The following proposition demonstrates that it suffices to show $X(\R^n; \R^n) \subset L^2(\R^n; \R^n)$ is a locally compact embedding.

\begin{proposition}[Compactness]\label{jIProperties} 
If $X(\R^n; \R^n) \subset L^2(\R^n; \R^n)$ is a locally compact embedding, then for every bounded and open $\Om \subset \R^n$, and every $\de>0$, the embedding $X_0(\Om_{\de}; \R^n) \subset L^2(\Om; \R^n)$ is compact.
\end{proposition}
\begin{proof}
As we remarked earlier, for every  $u\in X_0(\Om_{\de};\R^n)$, its extension by zero outside of $\Om_{\de}$ belongs to $X(\R^n;\R^n)$. Moreover, $[u]_{X(\Om_{\de}; \R^n)} = [u]_{X(\R^n; \R^n)}$. Now if the inclusion $\mathfrak{i}:X(\R^n;\R^n) \subset L^2(\R^n;\R^n)$ is locally compact, then in Definition \ref{localCompact}, we can set $K := \overline{\Om}$ to conclude that $R_K \mathfrak{i} : X(\R^n;\R^n) \to L^2(\Om;\R^n)$ is compact. The result now follows easily.  
\end{proof}

We now prove the local compact embedding of $X(\R^n; \R^n)$ in the remaining portion of this section. We follow the argument in  \cite{Jarlocalcompactness}.

\begin{lemma}[Convolution]\label{localcompactvecfield}
Suppose $W \in L^1(\R^n;\R^{n\times n})$ is a matrix-valued function with $L^1$-entries. Then the corresponding convolution operator $T_W: L^2(\R^n; \R^n) \rightarrow L^2(\R^n; \R^n)$ defined via
\begin{equation}\label{localCompactConvolution}
    [T_W u(x)]_{i} =  [(W * u)(x)]_{i} := \ \int_{\R^n} [W(x-y)]_{i,\cdot}\cdot u(y)dy =\sum_{j=1}^{n}\int_{\R^n} [W(x-y)]_{i,j} [u(y)]_jdy
\end{equation}
for each $i \in \{1, 2, \dots, n\}$ is locally compact.
\end{lemma}
\begin{proof}
The proof follows from \cite[Lemma 3.1]{Jarlocalcompactness}  after noting that for $i=1, 2, \dots, n$,  $[T_W u]_{i}$ is a finite sum convolution operators  which are locally compact.   
\end{proof}

\begin{theorem}[Local compactness]\label{Djcompact}
Fix $\de >0$. 
Suppose that $\frac{k_{\de}(\xi)}{|\xi|^2} \notin L^1(\R^n)$, then the space  
$X(\R^n; \R^n)$ is locally compactly embedded in $L^2(\R^n; \R^n).$
\end{theorem}
\begin{proof}
For $\tau > 0$, let $j_{\tau}(\xi) := \frac{k_{\de}(\xi)}{|\xi|^2} \mathds{1}_{\R^n\setminus{B(0, \tau)}}(\xi)$. Then  $j_{\tau} \in L^1(\R^n)$ and that, by assumption on $k_\de$, we have that $\|j_{\tau}\|_{L^1(\R^n)} \to \infty $ as $\tau\to 0$. We now introduce the  matrix-valued function 
\begin{equation}\label{DjcompactEq1}
    J_{\tau}(\xi) \ := c_n \ \frac{j_{\tau}}{\|j_{\tau}\|_{L^1(\R^n)}} \frac{\xi\otimes \xi}{|\xi|^2}.
\end{equation}
where $c_n$ is a normalizing constant that depends only on $n$ so that 
\begin{equation}\label{DjcompactEq2}
    \int_{\R^n}J_{\tau}(\xi)d\xi \ = \ \mathbb{I}_{n}, \quad \text{the identity matrix.}
\end{equation}
Let $u \in X(\R^n; \R^n)$, and we claim that
\begin{equation}\label{DjcompactEq3}
    \|u - T_{j_{\tau}}u\|_{L^2(\R^n;\R^n)} \ \leq \ \left(\frac{1}{\|j_{\tau}\|_{L^1(\R^n)}}\right)^{\frac{1}{2}}[u]_{X(\R^n; \R^n)}.
\end{equation}
We prove this via a direct calculation: rewrite $u - T_{j_{\tau}}u$ as
\begin{equation}\label{DjcompactEq3A}
    u(x) - T_{j_{\tau}}(u)(x) \ = \ \int_{\R^n}J_{\tau}(\xi)(u(x) - u(x + \xi))d\xi.
\end{equation}
Now, we calculate the $L^2(\R^n; \R^n)$-norm, and estimate it with the Cauchy-Schwarz Inequality and the pointwise inequality $j_{\tau}(\xi) \leq \frac{k_{\de}(\xi)}{|\xi|^2}$:
\begin{equation}\label{DjcompactEq3B}
    \begin{aligned}
        &\|u - T_{j_{\tau}}u\|^2_{L^2(\R^n; \R^n)} \ = \ \int_{\R^n}\left|\int_{\R^n}\frac{j_{\tau}( \xi)}{\|j_{\tau}\|_{L^1(\R^n)}}\left((u(x) - u(x + \xi)) \cdot \frac{\xi}{|\xi|}\right)\frac{\xi}{|\xi|} d\xi\right|^2 dx \\
        &\ \leq \ \frac{1}{\|j_{\tau}\|_{L^1(\R^n)}}\iint_{\R^{2n}}j_{\tau}(\xi)\left|(u(x) - u(x + \xi)) \cdot \frac{\xi}{|\xi|}\right|^2d\xi dx 
        \ \leq \ \frac{1}{\|j_{\tau}\|_{L^1(\R^n)}}[u]^2_{X(\R^n; \R^n)}.
    \end{aligned}
\end{equation}
Taking square roots in \eqref{DjcompactEq3B} immediately yields \eqref{DjcompactEq3}. 

Now let $M \subset X(\R^n; \R^n)$ be a bounded set, and $K \subset \R^n$ be compact; our proof will be complete once we show that $R_K(M) \subset L^2(\R^n; \R^n)$ is relatively compact. To this end, let $C := \sup_{u\in M}\|u\|_{X(\R^n; \R^n)}$ and $\ep > 0$. Since $\frac{k_{\de}(\xi)}{|\xi|^2} \notin L^1(\R^n)$, we may take $\tau > 0$ to be sufficiently small so that $\|j_{\tau}\|_{L^1(\R^n)} \geq \frac{C^2}{\ep^2}$. 
By Lemma \ref{localcompactvecfield}, the set $\til{M} := [R_KT_{j_{\tau}}](M)$ is relatively compact in $L^2(\R^n; \R^n)$. Thus we may use the estimate \eqref{DjcompactEq3} to obtain, for any $u \in M$, 
\begin{eqnarray}\label{DjcompactEq6}
    \|R_Ku - [R_KT_{j_{\tau}}]u\|_{L^2(\R^n; \R^n)} \ \leq \ \|u - T_{j_{\tau}}u\|_{L^2(\R^n; \R^n)} \ &\leq \ \\
    \left(\frac{1}{\|j_{\tau}\|_{L^1(\R^n)}}\right)^{\frac{1}{2}}[u]_{X(\R^n; \R^n)} \ \leq \ \frac{\ep\|u\|_{X(\R^n; \R^n)}}{C} \ &\leq \ \ep.
\end{eqnarray}
From this we conclude that $R_K(M)$ is contained within an $\ep$-neighborhood of $\til{M}$, and which is relatively compact in $L^2(\R^n; \R^n)$ (since $j_{\tau} \in L^1(\R^n)$). Thus, $R_K(M)$ is totally bounded in $L^2(\R^n; \R^n)$, which is a sufficient condition for the local compact embedding to hold.
\end{proof}

\begin{remark} We make two remarks. First, the assumption $j_0={k_{\de}(\xi)\over |\xi|^{2}} \notin L^1(\mathbb{R}^{n})$ cannot be waived. Indeed, otherwise, we have $X(\R^n; \R^n) = L^2(\R^n; \R^n)$ with the norm estimate that $[u]_{X(\R^n; \R^n)}^{2} \leq 4\|j_0\|_{L^1}\|u\|^2_{L^2(\R^n; \R^n)}$. 
Second, a similar type of compactness result, with a proof that uses a different approach (see \cite{braides2002gamma}),  is also established in \cite{mengesha2014bond} under the assumption on the kernel $k_{\de}$ that 
\begin{equation}\label{DMT-CC}
\lim_{\varrho \to 0}{\varrho^{2} \over \int_{B_{\varrho}(0)} k_{\de}(\xi)d\xi } \ = \ 0.
\end{equation}
On the one hand, if  $k_\de$ satisfies \eqref{DMT-CC}, then  ${k_{\de}(\xi)\over |\xi|^{2}} \notin L^1(\mathbb{R}^{n})$. Otherwise, by continuity of the integral, 
\begin{equation}\label{ctyIntegralKern}
\lim_{\varrho\to 0}\int_{B_{\varrho}(0)} {k_{\de}(\xi)\over |\xi|^{2}} d\xi \ = \ 0,
\end{equation}
from which it follows that $\lim_{\varrho\to0}{\varrho^2\over \int_{B_\varrho(0)} k_{\de}(\xi) d\xi} = \infty,$ contradicting \eqref{DMT-CC}.  As such \eqref{DMT-CC} is a more restrictive assumption on $k_\de.$ On the other hand, there are kernels with the property that ${k_{\de}(\xi)\over |\xi|^{2}} \notin L^1(\mathbb{R}^{n})$ that fail to satisfy \eqref{DMT-CC}. For example, the kernel $k_{\de}(\xi) = {1 \over |\xi|^{n-2}}\chi_{B_{\de}(0)}(\xi)$  has the property that $\lim_{\varrho \to 0}{\varrho^{2} \over \int_{B_{\varrho}(0)} k_{\de}(\xi)d\xi } >0$, yet ${k_{\de}(\xi)\over |\xi|^{2}} =  {1 \over |\xi|^{n}}\chi_{B_{\de}(0)}(\xi) \notin L^1(\mathbb{R}^{n})$. 
\end{remark}


\section{Well-posedness: state system and minimization}\label{oc}

In this section we show existence and uniqueness of solutions for each one of the optimal control problems introduced in Section \ref{sec:Notation}. The approach we use is a reduced formulation where the constrained optimization is reformulated as an unconstrained optimization of the control via the solution operator of the  state equation. To facilitate that we begin by proving an abstract well-posedness result that appears in some form in  \cite{hinze2008optimization, troltzsch2010optimal}; we provide a proof for the sake of completeness.

\begin{theorem}[Well-posedness]\label{abstractWP}
Let $(Y, \|\cdot\|_Y)$ be a real Banach space with $L^2(\Om;\R^n) \subset Y^*$. Suppose also that  $S: L^2(\Om; \R^n) \rightarrow Y$ is a compact operator, and $G: Y \rightarrow \R$ is lower semi-continuous. For a given $\la \geq 0$ and $Z_{\text{ad}}$ a nonempty, closed, bounded, and convex subset of $L^2(\Om; \R^n)$, define $j: Z_{\text{ad}} \rightarrow \R$ by
\begin{equation}\label{bondAbstractCostFunctional}
  j(g) \ := \ G(Sg) + \frac{\la}{2}\int_{\Om}\Ga(x) |g(x)|^2 dx, 
\end{equation}
for some non-negative $\Ga\in L^1(\Om)$. Then, the optimization problem
\begin{equation}\label{HilbertOptCompact}
    \min_{g \in Z_{\text{ad}}} j(g) 
\end{equation}
has a solution $\til{g}$. Furthermore, if $\la > 0$, $S$ is linear, and $G$ is convex, then \eqref{HilbertOptCompact} has a unique minimizer. 
Alternatively, if $\la=0$ and  $G$ is strictly convex on its domain (with $S$ still being linear), then the minimizer is unique.
\end{theorem}
\begin{proof}
We use the direct method of calculus of variations to show that \eqref{HilbertOptCompact} has a solution.  First, we note that $j$ is bounded from below. Indeed, since the second term is nonnegative for all $g$, it suffices to demonstrate that the first term is bounded from below. 
To that end, assume otherwise. 
Then there exists a sequence $\{w_m\}^{\infty}_{m = 1} \subset Z_{\text{ad}}$ such that
\begin{equation}\label{gCompactlscEq1}
    G(S w_m) \ < \ -m
\end{equation}
for all $m \in \N^+$. However, $Z_{\text{ad}}$ is a closed, bounded, convex subset of a Hilbert space, and, by \cite[Theorem 2.11]{troltzsch2010optimal}, it is weakly sequentially compact. It follows that some sub-sequence $\{w_{m_k}\}^{\infty}_{k = 1}$ of $\{w_m\}^{\infty}_{n = 1}$ converges weakly to some $\overline{w} \in Z_{\text{ad}}$. Since $S$ is a compact operator, $Sw_{m_k} \rightarrow Sw$ as $k \rightarrow \infty$ strongly in $Y$. Since $G$ is lower semi-continuous, we have
\begin{equation}\label{gCompactlscEq2}
  \ G(S\overline{w})\leq  \ \liminf_{k \rightarrow \infty}G(Sw_{m_k}) \ = -\infty,
\end{equation}
which poses a contradiction, since $G$ does not assume the value $-\infty$.  

We henceforth denote $j_0 := \inf_{g \in Z_{\text{ad}}}j(g)$, and the remainder of the existence part of the proof is comprised of finding $\til{g} \in Z_{\text{ad}}$ such that $j(\til{g}) = j_0$. To this end, we identify a sequence $\{g_m\}^{\infty}_{m = 1} \subset Z_{\text{ad}}$ such that $\lim_{m \rightarrow \infty}j(g_m) = j_0$ as $m \rightarrow \infty$. Recalling again \cite[Theorem 2.11]{troltzsch2010optimal} we obtain that some sub-sequence $\{g_{m_k}\}^{\infty}_{k = 1}$ of $\{g_m\}^{\infty}_{m = 1}$ converges weakly in $L^2(\Om;\mathbb{R}^{n})$ to some $\til{g} \in Z_{\text{ad}}$. Moreover, since $|\sqrt{\Ga}(x) g_{m}(x)| \leq  |\sqrt{\Ga}(x)  b(x)|$ for all $m$,  the sequence $\{\sqrt{\Ga} \,g_{m}\}$ is uniformly bounded in $L^2(\Om;\mathbb{R}^{n})$ as well.  From this we may choose the sub-sequence  $\{g_{m_k}\}^{\infty}_{k = 1}$ so that  $\{\sqrt{\Ga}\,g_{m_k}\}^{\infty}_{k = 1}$ converges weakly in $L^2(\Om;\mathbb{R}^{n})$. By a density argument, it is easy to show that the weak limit has to be $\sqrt{\Ga}\,\tilde{g}$. 
Since $S$ is compact and $G$ is lower semi-continuous, we have the inequality chain
\begin{equation}\label{fminEq1}
\begin{aligned}
j(\til{g}) \ &= \ G(S\til{g}) + \frac{\la}{2}\int_{\Om}\Ga(x) |\tilde{g}(x)|^2 dx \\
\ &\leq \ \liminf_{k \rightarrow \infty}G(Sg_{m_k}) + \frac{\la}{2}\int_{\Om}\Ga(x) |\tilde{g}(x)|^2 dx\\
&\leq \ \liminf_{k \rightarrow \infty}G(Sg_{m_k}) + \liminf_{k \rightarrow \infty}\frac{\la}{2} \int_{\Om}\Ga(x) |g_{m_k}|^2 dx \\
&\leq \ \liminf_{k \rightarrow \infty}\left(G(Sg_{m_k}) + \frac{\la}{2}\int_{\Om}\Ga(x) |g_{m_k}|^2 dx\right) \ \leq \ \lim_{k \rightarrow \infty}j(g_{m_k}) \ = \ j_0.
\end{aligned}
\end{equation}
Since $\til{g} \in Z_{\text{ad}}$, it follows that $j(\til{g}) = j_0$, and we have found a minimizer. The proof of uniqueness under the given additional conditions is standard since $j$ will automatically become strictly convex.
\end{proof}

\begin{corollary}[Existence and uniqueness]\label{fullExisteceAndUniqueness}
Problems \ref{nlctsProb}, \ref{nlDiscProb}, \ref{lCtsProb}, and \ref{lDiscProb} are all well-posed. That is, the objective functional has a [unique] minimizing pair, which in turn solves the corresponding state equation.
\end{corollary}
\begin{proof}
The well-posedness of the state equation of each problem follows from the Lax-Milgram lemma as done in Corollary \ref{Well-state}. Notice that in all cases, the solution space $Y$ is compactly embedded into $L^2(\Om;\R^n)$. For the local problems, the embedding $H^1_0(\Om; \R^n)\Subset L^2(\Om; \R^n)$ is standard, while for the non-local problems we invoke Theorem \ref{Djcompact} and Proposition \ref{jIProperties}. We thus have that the solution mapping $S: L^2(\Om;\R^n) \to Y$ is compact, and then 
we may write the reduced cost functionals for our problems abstractly as
\begin{equation}\label{optCondReducedCost}
    j(g) \ := \ \int_{\Om}F(x, Sg(x))dx + \frac{\la}{2}\int_{\Om}\Ga(x) |g(x)|^2 dx.
\end{equation}
Note that this functional satisfies all the conditions of Theorem \ref{abstractWP}, which guarantees  existence and uniqueness of a minimizer. 
\end{proof}

\section{Analysis in vanishing horizon parameter}\label{gaconv}
Having shown that, for every horizon $\de \geq 0$, the nonlocal optimal control problem \ref{nlctsProb} has a unique solution $(\overline{u}_\de, \overline{g}_\de)$, we now study the behavior of the pair as $\de\to 0$. Notice that $\overline{u}_\de$ minimizes 
the potential energy functional
\begin{equation}\label{pot-energy}
W_{\de}(u) \ := \ B_{\de}(u, u) - \int_{\Om} g_\de (x)\cdot u(x)dx
\end{equation}
over $X_0(\Om_{\de}; \R^n)$.  We begin with the following convergence result.   

\begin{lemma}[Compactness of solutions of the control problem]\label{convergenceStateControl}
Let $\{(\overline{u_{\de}}, \overline{g_{\de}})\}_{\de > 0}$ be the family of optimal state-control pairs solving  \ref{nlctsProb}. There exists a $(\overline{u}, \overline{g}) \in H^1_0(\Om; \R^n) \times Z_{\text{ad}}$ such that, up to a sub-sequence, $\overline{g}_{\de} \rightharpoonup \overline{g}$ in $L^2(\Om; \R^n)$ and $\overline{u_{\de}} \rightarrow \overline{u}$ strongly in $L^2(\Om; \R^n)$ as $\de \to 0^+$. 
\end{lemma}
\begin{proof} 
Theorem \ref{abstractWP} gives existence and uniqueness of optimal pairs that minimize the energy $W_\de$ defined in  \eqref{pot-energy}. 
Moreover, since $0$ is an admissible control, we have that 
 $W_{\de}(\overline{u_{\de}}) \ \leq \ 0, $ 
and so, after rearranging we get 
\begin{equation}\label{stateControlEq2}
B_{\de}(\overline{u_{\de}}, \overline{u_{\de}}) \ \leq \ \int_{\Om}\overline{g_\de}(x) \cdot \overline{u_{\de}}(x)dx.
\end{equation}
The Cauchy-Schwarz Inequality, in conjunction with the nonlocal Poincaré inequality \eqref{sharperPoincareIneq} and the Triangle inequality, gives us
\begin{equation}\label{stateControlEq3}
    [\overline{u_{\de}}]^2_{X(\Om_{\de}; \R^n)} \ \lesssim \ \|\overline{g_\de}\|_{L^2(\Om; \R^n)}\|\overline{u_{\de}}\|_{L^2(\Om; \R^n)} \lesssim \|\overline{g_\de}\|_{L^2(\Om; \R^n)}[\overline{u_{\de}}]_{X(\Om_{\de}; \R^n)}.
\end{equation}
Notice that the constant in this estimate, owing to \eqref{sharperPoincareIneq}, is independent of $\de$. Furthermore, since $\{\overline{g_{\de}}\}_{\de > 0} \subset Z_{\text{ad}}$, it is norm bounded (and therefore has a weak limit, up to a sub-sequence), and as a consequence  
\begin{equation}\label{stateControlEq4}
    \sup_{\de>0}[\overline{u_{\de}}]_{X(\Om_{\de}; \R^n)} \ \leq \ C.
\end{equation}
Now since $\overline{u}_\de \in X_0(\Om_{\de}; \R^n)$, after extending by zero to $\Om_1$ (with $\de=1$) we have that 
\[
\sup_{\de > 0}\int_{\Om_1}\int_{\Om_1} k_{\de}(x-y){|Du_{\de}(x,y)|^2 \over |x-y|^{2}} dydx  = \sup_{\de > 0}[\overline{u_{\de}}]_{X(\Om_{\de}; \R^n)}^2 \ \leq \ C.
\]
From this, we may use \cite[Proposition 4.2]{mengesha2015variational} or \cite[Theorem 2.5]{mengesha2012nonlocal} to conclude that the  $\{\overline{u_{\de}}\}_{\de > 0}$ is precompact in $L^2(\Om;\R^n)$ and converges strongly in $L^2(\Om; \R^n)$ to some $\overline{u} \in H^1_0(\Om; \R^n)$ (up to a sub-sequence). 
\end{proof}

The main question we would like to address in the remaining is whether the limiting pair $(\overline{u}, \overline{g})$ solves a corresponding  limiting optimal problem. 
The limiting behavior of the minimizers is closely related to the variational convergence of the above parametrized energy functionals.  
The main tool we shall use is $\Ga$-convergence (see \cite{braides2002gamma, buttazzo1982gamma, dal2012introduction, Rin} for more on properties of $\Ga$-convergence; \cite{bellido2015hyperelasticity, Bon, mengesha2015variational, mengesha2016characterization, Rin} for examples of proofs of $\Ga$-convergence for other peridynamics models. 
For convenience, we recall its definition here.

\begin{definition}[$\Ga$-convergence]\label{GaConvDef}
We say that the sequence $E_{\de}: L^2(\Om; \R^n) \rightarrow \R \cup \{+\infty\}$ \textbf{$\Ga$-converges} strongly in $L^2(\Om; \R^n)$ to $E_0: L^2(\Om; \R^n) \rightarrow \R \cup \{+\infty\}$ (denoted $E_{\de} \xrightarrow{\Ga} E_0$) if the following properties hold:
\begin{enumerate}[label=\textbf{GC\arabic*}]
  \item\label{GC1} \textbf{The liminf property:} Assume $u_{\de} \rightarrow u$ strongly in $L^2(\Om; \R^n)$. Then we have the Fatou-type inequality
    \begin{equation}\label{FatouOfEDelta}
        E_0(u) \ \leq \ \liminf_{\de \rightarrow 0^+}E_{\de}(u_{\de}).
    \end{equation}

    \item\label{GC2} \textbf{Recovery sequence property:} For each $u \in L^2(\Om; \R^n)$, there exists a sequence $\{u_{\de}\}_{\de > 0}$ where $u_{\de} \rightarrow u$ strongly in $L^2(\Om; \R^n)$ and
    \begin{equation}\label{RecoverySeqProperty}
        \limsup_{\de \rightarrow 0^+}E_{\de}(u_{\de}) \ \leq \ E_0(u).
    \end{equation}
\end{enumerate}
\end{definition}

\subsection{Vanishing horizon parameter for continuous problem}\label{gaconvConts}
We will be working on the extended linear peridynamic energy functional we now define. 
Let $E_{\de}: L^2(\Om_{\de}; \R^n) \rightarrow [0, \infty]$ denote the energy
\begin{equation}\label{energyEdeReprise}
    E_{\de}(u) \ := \ \iint_{\mathcal{D}_{\de}}H(x, y)k_{\de}(x - y)\frac{|Du(x, y)|^2}{|x - y|^2}dxdy,\quad \text{for $u\in X(\Om_{\de}; \R^n)$}
\end{equation}
and $+\infty$ otherwise. Similarly, define a limiting energy $E_0: L^2(\Om; \R^n) \rightarrow [0, \infty]$ by 
\begin{equation}\label{energyE0}
    E_0(u) \ := \ \frac{1}{n(n + 2)}\int_{\Om}H(x,x)(2\|\Sym(\grad u(x))\|^2_F + \dive(u(x))^2)dx \quad \text{for $u\in H^1(\Om; \R^n)$,}
\end{equation}
and $+\infty$ otherwise.  
Note that since, for all $\de > 0$, our energy $E_\de$ is quadratic we have
 \begin{equation}\label{EnergyDeDiffSqEq}
    |E_{\de}(u) - E_{\de}(v)| \ \leq \ E_{\de}(u + v)^{\frac{1}{2}}E_{\de}(u - v)^{\frac{1}{2}}
\end{equation}
for all $u, v \in X(\Om_{\de}; \R^n)$.
\begin{lemma}[Nonlocal to local] \label{L-to-NL-for-smooth}
Suppose that $A\Subset\Om$ and $w\in C^{2}(A,\R^n)$. Then for any $h\in L^\infty(\Om)$, we have that  
\[
\begin{split}
    \lim_{\de \rightarrow 0^+}\int_A\int_{\Om}&h(x)k_{\de}(x - y)\frac{|Dw(x, y)|^2}{|x - y|^2}dydx \ \\
    &= 
    \frac{1}{n(n + 2)}\int_A h(x)(2\|\Sym(\grad w(x))\|^2_F + \dive(w(x))^2)dx.
    \end{split}
    \]
\end{lemma}
The proof of this can be found in \cite{du2013analysis,mengesha2014bond,mengesha2015variational} in some form or another. 

We now state the result on the variational convergence of the parameterized energies $E_{\de}$. 

\begin{theorem}[$E_{\de} \xrightarrow{\Ga} E_0$]
\label{GaConvEde}
Let $E_\de$ and $E_0$ be defined in \eqref{energyEdeReprise} and \eqref{energyE0}, respectively. We have $E_{\de} \xrightarrow{\Ga} E_0$ in the sense of in Definition \ref{GaConvDef}. 
\end{theorem}
\begin{proof} 
We verify each of the conditions that comprise this definition.

\textbf{Proof of \ref{GC1}:} Let $u \in L^2(\Om;\R^n)$ be arbitrary, and $\{u_{\de}\}_{\de > 0} \subset L^2(\Om; \R^n)$ be such that $u_{\de} \rightarrow u$ strongly in $L^2(\Om; \R^n)$; we may assume without loss of generality that $\liminf_{\de \rightarrow 0^+}E_{\de}(u_{\de}) <  \infty$. That is, up to a sub-sequence we may assume that $E_{\de}(u_{\de}) <  \infty$ 
and using the positive lower bound on the coefficient $H$ we have that 
\begin{equation}\label{liminfIneqProofEq2}
\sup_{\de > 0}\int_{\Om}\int_{\Om}\frac{k_{\de}(x - y)}{|x - y|^2}|Du_{\de}(x, y)|^2dxdy \ < \ \infty.
\end{equation}
Arguing in the same way as in the proof of \cite[Theorem 2.5]{mengesha2012nonlocal}, we then have $u \in H^1(\Om; \R^n)$, and that $u_{\de} \rightarrow u$ strongly in $L^2(\Om; \R^n)$.  

From here we will look to find a variant of \cite[Equation 37]{mengesha2015variational}, largely repeating the lower semi-continuity part of the proof of \cite[Theorem 4.4]{mengesha2016characterization}. We first assume that $h$ is the constant function $h=1$ and prove that for any $A\Subset \Om$ open, we have the inequality
\begin{equation}\label{liminfhIndicator}
    \frac{1}{n(n + 2)}\int_A 2\|\Sym(\nabla u(x))\|^2_F + \dive(u(x))^2dx \leq \liminf_{\de \rightarrow 0^+}\int_A\int_{\Om_{\de}}\frac{k_{\de}(x - y)}{|x - y|^2}|Du_{\de}(x, y)|^2dxdy.
\end{equation}
Let $0 < \ep < \dist(A, \pa \Om)$, and let $\eta \in C^{\infty}_0(B(0, 1))$ be a smooth cutoff function. Define $\eta_{\ep}(z) := \ep^{-n}\eta\left(\frac{z}{\ep}\right)$, and define $w_{\ep, \de} := \eta_{\ep} * u_{\de}$ on $\overline{A}$, which is in $C^{2}(\bar{A};\R^n)$. 
Via a direct calculation coupled with application of Jensen's inequality, we have 
\begin{equation}\label{liminfIneqProofEq9}
    \int_{A}\int_{\Om}k_{\de}(x - y)\frac{|Dw_{\ep, \de}(x, y)|^2}{|x - y|^2}dxdy \ \leq \ \int_{\Om}\int_{\Om_{\de}}k_{\de}(x - y)\frac{|Du_{\de}(x, y)|^2}{|x - y|^2}dxdy.
\end{equation}  
Our next step will be to send $\de \rightarrow 0^+$, leaving $\ep > 0$ fixed for now. The right hand side of \eqref{liminfIneqProofEq9} is bounded by $\liminf_{\de \rightarrow 0^+}E_{\de}(u_{\de})$ (with $h=1$). We compute the limit of the left hand  side. Set $w_{\ep} := \eta_{\ep} * u$. Then we observe that $w_{\ep, \de} \rightarrow w_{\ep}$ as $\de \rightarrow 0^+$ in $C^1(\overline{A}; \R^n)$ due to $u_{\de} \rightarrow u$ in $L^2(\Om; \R^n)$ (where $\ep > 0$ is taken to be fixed for now). 
We  use this and Lemma \ref{L-to-NL-for-smooth} to obtain that  
\begin{equation}\label{liminfIneqProofEq29}
    \frac{1}{n(n + 2)}\int_A(2\|\Sym(\grad w_{\ep})\|^2_F + \dive(w_{\ep}(x))^2dx=\lim_{\de \rightarrow 0^+}\int_A\int_{\Om}k_{\de}(x - y)\frac{|Dw_{\ep, \de}(x, y)|^2}{|x - y|^2}dxdy. \ 
\end{equation}
The desired inequality \eqref{liminfhIndicator} now follows from taking the limit in $\epsilon$ in \eqref{liminfIneqProofEq29} and combining it with \eqref{liminfIneqProofEq9}.

Next we assume that $h$ is a simple function $h(x)=  \sum_{i=1}^{m} h_i \chi_{D_i}(x)$ for $D_i\subset \Om_{\de}$. Then applying \ref{liminfhIndicator} for each $i=1, \dots, m$ over $D_i\cap A$ and summing it over $i \in \{1, 2, \dots, m\}$ we have 
\[
\begin{split}
  \frac{1}{n(n + 2)}\int_A  h(x) & \, \left( 2\|\Sym(\nabla u(x))\|^2_F + \dive(u(x))^2 \right) dx\\
  & \leq \sum_{i = 1}^{m}\liminf_{\de \rightarrow 0^+}\int_{D_i\cap A}h_i\int_{\Om_{\de}}\frac{k_{\de}(x - y)}{|x - y|^2}|Du_{\de}(x, y)|^2dxdy\\
  &\leq \liminf_{\de \to 0^+}\int_{ A}h(x)\int_{\Om_{\de}}\frac{k_{\de}(x - y)}{|x - y|^2}|Du_{\de}(x, y)|^2dxdy,
  \end{split}
\]
where we use the sub-additivity $\liminf a_\tau + \liminf b_\tau \leq \liminf(a_\tau  + b_\tau)$.  
Finally, the case of general positive $h \in L^{\infty}(\Om)$, we select an increasing sequence of step functions ${s_j(x)}$, $0\leq s_j \leq s_j+1 \leq h(x)$ that converges to $h$ uniformly.  
The result then follows from direct application of the Monotone Convergence Theorem. 

\textbf{Proof of \ref{GC2}:} Let $u\in L^2(\Om;\R^n)$. We may assume that $u \in H^1(\Om; \R^n)$. For the recovery sequence, we take $u_{\de} := \tilde{u}\in H^{1}(\R^n, \R^{n})$, which is the extension of $u$ to $\R^{n}$ with compact support say in $\Om_1$ (with $\de=1$).   
Take a sequence $\{v_j\}^{\infty}_{j = 1} \subset C^{2}(\Om_1)$ such that  $v_j \to \tilde{u}$ n $H^{1}(\Om_1)$ as $j\to \infty.$ then using \eqref{EnergyDeDiffSqEq}, we see that for a $C>0$
\[
\sup_{\de>0} |E_{\de}(\tilde{u}) - E_\de(v_j)| \leq C \|\nabla u - \nabla v_j\|_{L^2(\Om; \R^{n \times n})}\|\nabla u\|_{L^2(\Om; \R^{n \times n})}. 
\]
That means, $E_\de(v_j) \to E_{\de}(\tilde{u})$ as $j \to \infty$, uniformly in $\de$.  Using the same proof as Lemma \ref{L-to-NL-for-smooth}, we see that for each $j=1, 2, \dots$,
\[
\lim_{\de\to 0} E_{\de} (v_j) \ = \ B_{0}(v_j, v_j).  
\]
Taking the limit in $j$ now we have that 
\[
\lim_{\de \to 0} E_{\de}(\tilde{u}) = \lim_{\de \to 0} \lim_{j\to \infty}E_{\de}(v_j) =  \lim_{j\to \infty}\lim_{\de \to 0} E_{\de}(v_j) = \lim_{j\to \infty}B_{0}(v_j, v_j) = B_{0}(u, u) = E_0(u),
\]
where in the second equality we used the uniform convergence in $\de$. 
\end{proof}

\begin{remark}
We may follow the above approach as well as \cite[Remark 1.7]{braides2002gamma} to conclude that the family of energies $\{W_\de\}_{\de>0}$, defined in \eqref{pot-energy} (finite on $X_0(\Om;\R^n)$), also $\Ga$-converges in the $L^2$-topology to 
\[
W_{0}(u) = B_0(u, u)-\int_{\Om} \overline{g}(x) \cdot u(x) dx
\] 
where $g_\de \rightharpoonup \overline{g}$ weakly in $L^2(\Om; \R^n)$ as $\de \to 0^+$. 
With $\Ga$-convergence at hand, we recall that \cite[Corollary 7.20]{dal2012introduction} states if $\{u_{\de}\}_{\de > 0}$ is a family of minimizers for $\{W_{\de}\}_{\de > 0}$ over $L^2(\Om; \R^n)$, and ${\overline{u}}$ is a limit point of this family, then $\overline{u}$ is a minimizer of $W_0$ on $L^2(\Om; \R^n)$ (see also \cite[Theorem 2.1]{buttazzo1982gamma}). By our previous results, this implies $\overline{u} \in H^1_0(\Om; \R^n)$, and moreover
\begin{equation}\label{convergenceOfMinimizers}
    W_0(\overline{u}) \ = \ \lim_{\de \rightarrow 0^+}W_{\de}(u_{\de}).
\end{equation}
\end{remark}
Finally, we identify what conditions to impose to identify the solution to the local optimal control problem via a limiting process.

\begin{theorem}[Convergence]\label{convOfSolutions}
Suppose $\{(\overline{u_{\de}}, \overline{g_{\de}})\}_{\de > 0} \in \mathcal{A}^\de$ is the family of solutions to the nonlocal continuous Problem \ref{nlctsProb}. Then, there exists $(\overline{u}, \overline{g}) \in \mathcal{A}^{loc}$ such that, up to a non-relabeled sub-sequence, $\overline{u_{\de}} \rightarrow \overline{u}$ in $L^2(\Om; \R^n)$ and $\overline{g_{\de}} \rightharpoonup \overline{g}$ in $L^2(\Om; \R^n)$. Moreover, $(\overline{u}, \overline{g})$ solves the local optimal control Problem \ref{lCtsProb}.
\end{theorem}

\begin{proof}
Lemma \ref{convergenceStateControl} gives the existence of such pair $(\overline{u},\overline{g}) \in \mathcal{A}^{loc}$. We now need to show that this pair minimizes the reduced objective functional in Problem \ref{lCtsProb}. Let $(v, f) \in \Alg^{\loc}$ be arbitrary, and consider, for $\de>0$ the sequence $(v_{\de}, f) \in \mathcal{A}^\de$, i.e., of solutions to the nonlocal boundary value problem \eqref{stateSystem}. We can repeat our argument from Lemma \ref{convergenceStateControl} with $g_\de = g = f$, and see that $v_{\de} \rightarrow v$ strongly in $L^2(\Om; \R^n)$. Then, by the Dominated Convergence Theorem, we have that 
\begin{equation}\label{convOfSolnEq1}
\begin{aligned}
    I(v, f) \ &= \ \int_{\Om}F(x, v(x))dx + \frac{\la}{2}\int_{\Om}\Ga(x) |f(x)|^2 dx  \\
    \ &= \ \lim_{\de \rightarrow 0^+}\left(\int_{\Om}F(x, v_{\de}(x))dx + \frac{\la}{2}\int_{\Om}\Ga(x) |f(x)|^2 dx\right) \\
    \ &= \ \lim_{\de \rightarrow 0^+}I(v_{\de}, f).
\end{aligned}
\end{equation}
Now we observe that $\lim_{\de \rightarrow 0^+}I(v_{\de}, f) \ \geq \ \lim_{\de \rightarrow 0^+}I(\overline{u_{\de}}, \overline{g_{\de}})$ since $\{(\overline{u_{\de}}, \overline{g_{\de}})\}_{\de > 0}$ was chosen as the minimizers for the objective functional \eqref{dmethod2CostIntro}. Next, notice that $\lim_{\de \rightarrow 0^+}I(\overline{u_{\de}}, \overline{g_{\de}}) \ \geq \ I(\overline{u}, \overline{g})$ due to Fatou's Lemma, where we recall that strong $L^2(\Om; \R^n)$ convergence of $\overline{u_{\de}} \rightarrow \overline{u}$ implies a.e. convergence in $\Om$. In summary, the inequality chain 
\begin{equation}\label{limInfAgain}
    I(v, f) \ = \ \lim_{\de \rightarrow 0^+}I(v_{\de}, f) \ \geq \ \lim_{\de \rightarrow 0^+}I(\overline{u_{\de}}, \overline{g_{\de}}) \ \geq \ I(\overline{u}, \overline{g})
\end{equation}
concludes the proof.
\end{proof}


\subsection{Vanishing horizon parameter for discrete problem}\label{gaconvDiscrete}

In order to establish the asymptotic compatibility in Section \ref{compatibility}, one must also consider the $\Ga$-convergence of the discrete problem. The course of proof is similar to that of $\Ga$-convergence for the continuous problem, but one can use the fact that $X_h \subset W^{1, \infty}_0(\Om; \R^n) \subset H^1_0(\Om; \R^n)$ to avoid the use of mollifiers. For these reasons, we merely state the results.

\begin{proposition}[$\Ga$-convergence of discrete problems]\label{convDiscreteGammaTheorem}
We have that $W_{\de} \xrightarrow{\Ga} W_0$ in the family of spaces $\{X_{\de, h}\}_{\de > 0}$ in the strong $L^2(\Om; \R^n)$ topology.
\end{proposition}

We also present the discrete analogue to to Theorem \ref{convOfSolutions}.

\begin{theorem}[Discrete convergence]\label{discreteConvToSoln}
Suppose $\{(\overline{u_{\de, h}}, \overline{g_{\de, h}}\}_{\de > 0} \in \mathcal{A}^{\de}_h$ is the family of solutions to the non-local discrete problem \ref{nlDiscProb}. Then, there is $(\overline{u_h}, \overline{g_h}) \in \mathcal{A}^{\loc}_h$ such that $\overline{u_{\de, h}} \rightarrow \overline{u_h}$ in $L^2(\Om; \R^n)$ and $\overline{g_{\de, h}} \rightharpoonup \overline{g_{\de}}$ in $L^2(\Om; \R^n)$. Moreover, $(\overline{u_h}, \overline{g_h})$ solves the local discrete optimal control Problem \ref{lDiscProb}.
\end{theorem}


\section{First order Optimality and discretization}\label{FEMAnalysis}
Let us now turn our attention to first-order optimality conditions, which are the gateway to discretizing the nonlocal optimal control problem. From here onward, we assume that our integrand $F$ (first introduced in \eqref{dmethod2CostIntro}) is continuously Gâteaux-differentiable in the second argument. The first Gâteaux derivative will be denoted as $F_u$. We will also denote by $S_{\de}$ the solution operator corresponding to the state system \eqref{stateSystem}, and by $S_{\de}^*$ the adjoint of $S_{\de}$ in the $L^2$-sense. Due to Corollary \ref{fullExisteceAndUniqueness}, the operator $S_{\de}$ is well defined. Using the reduced objective functional \eqref{optCondReducedCost}, we recall that \cite[Lemma 2.21]{troltzsch2010optimal} shows the first order necessary condition
\begin{equation}\label{nlCTS1stOrderOptCond}
  \lang j'(\overline{g_{\de}}), \ga_z - \overline{g_{\de}}\rang \ \geq \ 0 \ \quad \fa \ga_z \in Z_{\text{ad}},
\end{equation}
where $j'$ represents the derivative of $j$ in some appropriate sense. 
 This functional has two terms that need to be differentiated: for the first term, we use the Fréchet differentiability of $F$ and the Chain Rule; the derivative of the second term comes from the Fréchet derivative of $\|\cdot\|^2_{L^2_{\Ga}(\Om; \R^n)}$ (the weighted $\Ga$ norm). See \cite[Lemma 3.5]{d2019priori} for a similar calculation corresponding to the fractional Laplacian.  
 Inequality \eqref{nlCTS1stOrderOptCond} can now be rewritten as 
\begin{equation}\label{optCondVarIneq}
    \lang j'(\overline{g_{\de}}), \ga_z - \overline{g_{\de}}\rang \ = \ \Big\lang S_{\de}^*F_u(\cdot, S_{\de}\overline{g_{\de}}(\cdot)) + \la \Ga\,\overline{g_{\de}}, \ga_z - \overline{g_\de}\Big\rang \ \geq \ 0 \ \quad \fa \ga_z \in Z_{\text{ad}}.
\end{equation}
It is standard to introduce a new notation to rewrite the above as the system 
\begin{eqnarray}\label{OptCondEq}
\begin{aligned}
&\lang \overline{p_{\de}} + \la \Ga\,\overline{g_{\de}}, \ga_z - \overline{g_{\de}}\rang \ \geq \ 0, \quad \forall \ga_z \in Z_{\text{ad}} \\
&\overline{p_{\de}} \ = \ S_{\de}^*F_u(\cdot, \overline{u_{\de}}) \\
&\overline{u_{\de}} \ = \ S_{\de}\overline{g_{\de}}.
\end{aligned}
\end{eqnarray}
Note that $S_{\de}$ is a self-adjoint operator, so $S_{\de}^*F_u(\cdot, \overline{u_{\de}} ) \ = \ S_{\de}F_u (\cdot, \overline{u_{\de}})$, and so $\overline{p_\de} \in X_0(\Om_{\de}; \R^n)$. Furthermore, as a consequence of these conditions, in the event that $\Ga=1$, we obtain  $g_{\de}$ is the $L^2$-projection of the adjoint $\overline{p_\de}$ onto the control space $Z_{\text{ad}}$, i.e. 
\begin{equation}\label{nlCtsCtrlToAdjoint}
  \overline{g_{\de}}(x) \ = \ -\frac{1}{\la}\P_{Z_{\text{ad}}}(\overline{p_{\de}}(x)),
\end{equation}
where $\P_E$ denotes the $L^2$-projection onto the set $E$. Notice that, owing to the assumption that the objective functional is strictly convex, these first order necessary conditions are also sufficient. We summarize the result as follows. 

\begin{proposition}[Optimality conditions]\label{projectionNLCtsSum}
For every $\de>0$, the pair $(\overline{u_{\de}}, \overline{g_{\de}}) \in X_0(\Om_{\de}; \R^n)\times Z_{ad}$ is a solution to  Problem \ref{nlctsProb} if and only if \eqref{OptCondEq} holds. 
\end{proposition}

\subsection{Error analysis for nonlocal problems}\label{refineMeshNL}

With the aid of the optimality system, we are able to perform an error analysis, which we now begin. From here on we assume, for simplicity, that $\Ga \equiv 1$ and that $F(x,v) = \tfrac12 |v|^2$.  With this at hand, the optimality conditions for the non-local discrete problem read:
\begin{eqnarray}\label{OptCondDEq}
\begin{aligned}
\lang \overline{p_{\de, h}} + \la \overline{g_{\de, h}}, \ga_h - \overline{g_{\de, h}}\rang \ &\geq \ 0, \quad \forall \ga_h \in Z_{\text{ad}} \cap Z_h\\
\overline{p_{\de, h}} \ &= \ S_{{\de, h}}^*\overline{u_{\de, h}} ,\\
\overline{u_{\de, h}} \ &= \ S_{{\de, h}}\overline{g_{\de, h}},
\end{aligned}
\end{eqnarray}
where $S_{{\de, h}}$ is the discrete solution operator, and $S_{{\de, h}}^*$ is its discrete $L^2$ adjoint. Note that $S_{{\de, h}}$ is a self-adjoint operator, so $S^*_{\de, h}F_u(\cdot, \overline{u_{\de, h}} ) = S_{{\de, h}}F_u (\cdot, \overline{u_{\de, h}} )$. Also as with the non-local continuous optimality conditions, it follows that $\overline{g_{\de, h}}(x) \ = \ -\frac{1}{\la}\P_{Z_{\text{ad}}}(\Pi_0 \overline{p_{\de, h}}(x))$, where $\Pi_0: L^2(\Om; \R^n) \rightarrow Z_h$ denotes the $L^2$-projection onto $Z_h$.

To ease the error analysis, define the intermediary functions $\widehat{u_{\de}}, \widehat{p_{\de}} \in X_0(\Om_{\de}; \R^n)$ such that
\begin{equation}\label{discretizeEq3A}
        B_{\de}(\widehat{u_{\de}}, v_{\de}) \ = \ \lang \overline{g_{\de, h}}, v_{\de}\rang \ \quad \fa v_{\de} \in X_0(\Om_{\de}; \R^n); 
        \end{equation}
        \begin{equation}\label{discretizeEq3B}
        B_{\de}(v_{\de}, \widehat{p_{\de}}) \ = \ \lang v_{\de}, \overline{u_{\de, h}}\rang \ \quad \ \fa v_{\de}\in X_0(\Om_{\de}; \R^n).
\end{equation}
The existence and uniqueness of these functions follows from the Lax-Milgram Theorem. More importantly, we observe that the optimal discrete state and adjoint variables are nothing but the Galerkin approximations to $\widehat{u_{\de}}, \widehat{p_{\de}}$, respectively. From this we immediately obtain, using C\'ea's Lemma,  that
\begin{equation}\label{eq:BestApprox}
  \begin{aligned}
    \|\widehat{u_{\de}} - \overline{u_{\de, h}}\|_{X(\Om_{\de}; \R^n)} &\lesssim \inf_{v_{\de, h} \in X_{\de,h} } \|\widehat{u_{\de}} - v_{\de, h} \|_{X(\Om_{\de}; \R^n)}, \\
    \|\widehat{p_{\de}} - \overline{p_{\de,h}}\|_{X(\Om_{\de}; \R^n)} &\lesssim \inf_{q_{\de, h} \in X_{\de,h} } \|\widehat{p_{\de}} - q_{\de, h} \|_{X(\Om_{\de}; \R^n)}.
  \end{aligned}
\end{equation}

We now prove error estimates for the state and adjoint.

\begin{theorem}[State and adjoint error estimates]\label{FEMErrorEst} Suppose that $(\overline{u_{\de, h}}, \overline{g_{\de, h}})$ is the solution to Problem \ref{nlDiscProb}; $\overline{p_{\de, h}}$ solves the discrete adjoint equation in \eqref{OptCondDEq} given $\overline{u_{\de, h}}$; ($\overline{u_{\de}}, \overline{g_{\de}})$ is the solution to Problem \ref{nlctsProb}; and $\overline{p_{\de}}$ solves the continuous adjoint equation in \eqref{OptCondEq} corresponding to the state $\overline{u_{\de}}$. Then we have these error estimates for the states, and the adjoints:
\begin{equation}\label{FEMErrorEstEq1}
    \|\overline{u_{\de}} - \overline{u_{\de, h}}\|_{X(\Om_{\de}; \R^n)} \ \lesssim \ \inf_{v_{\de, h} \in X_{\de, h}}\|\widehat{u_{\de}} - v_{\de, h}\|_{X(\Om_{\de}; \R^n)} + \|\overline{g_{\de}} - \overline{g_{\de, h}}\|_{L^2(\Om; \R^n)};
\end{equation}
\begin{multline}\label{FEMErrorEstEq2}
    \|\overline{p_{\de}} - \overline{p_{\de, h}}\|_{X(\Om_{\de}; \R^n)} \ \lesssim \ \inf_{v_{\de, h} \in X_{\de, h}}\|\widehat{p_\de} - v_{\de, h}\|_{X(\Om_{\de}; \R^n)} + \inf_{v_{\de, h} \in X_{\de, h}}\|\widehat{u_{\de}} - v_{\de, h}\|_{X(\Om_{\de}; \R^n)} \\+ \|\overline{g_{\de}} - \overline{g_{\de, h}}\|_{L^2(\Om; \R^n)}.
\end{multline}
\end{theorem}
\begin{proof}
We begin by proving \eqref{FEMErrorEstEq1}. Substitute $v_{\de} := \overline{u_{\de}} - \widehat{u_{\de}}$ in \eqref{stateSystem} and \eqref{discretizeEq3A}, and subtract those two equations to obtain
\begin{equation}\label{FEMErrorEstEq3}
    B_{\de}(\overline{u_{\de}} - \widehat{u_{\de}}, \overline{u_{\de}} - \widehat{u_{\de}}) \ = \ \lang \overline{g_{\de}} - \overline{g_{\de, h}}, \overline{u_{\de}} - \widehat{u_{\de}}\rang.
\end{equation}
Using the definition of $H(x, y)$, Hölder's inequality, and \eqref{sharperPoincare} gives
\begin{equation}\label{FEMErrorEstEq7}
    \|\overline{u_{\de}} - \widehat{u_{\de}}\|_{X(\Om_{\de}; \R^n)} \ \lesssim \ \|\overline{g_{\de}} - \overline{g_{\de, h}}\|_{L^2(\Om; \R^n)}.
\end{equation}
This, combined with \eqref{eq:BestApprox} then yields the result.
The proof of \eqref{FEMErrorEstEq2} uses the same procedure, and is thus omitted. 
\end{proof}

At this stage we must observe that the infima in \eqref{FEMErrorEstEq1} and \eqref{FEMErrorEstEq2} tend to zero as $h \to 0^+$. This is because of density; if a rate of convergence in these terms is desired, then further regularity of $\widehat{u_{\de}}$ and $\widehat{p_\de}$ must be studied. For some kernels this could be done, for instance, by exploiting that $\overline{u_{\de,h}}$ belongs to a space that is strictly smaller than the dual of $X_0(\Om_{\de};\R^n)$; see, for instance, \cite{acosta2017fractional, grubb2015fractional,ros2014dirichlet}. Due to the generality we place on our kernel, we do not pursue this. It remains to estimate the difference between continuous and discrete controls, which will now be our focus. 

While in general our controls only belong to $L^2(\Om; \R^n)$, in the event we have additional regularity, we can quantify our forthcoming estimates even more. Indeed, in the local case, the projection formula $\overline{g}(x) = - \frac{1}{\la}\P_{Z_{\text{ad}}}(\overline{p}(x))$ combined with the fact that $\overline{p} \in H^1_0(\Om; \R^n)$ imply further regularity on the control (namely, that $\overline{g}\in H^1(\Om; \R^n)$). The following lemma provides a sufficient condition on the kernel for this to also be the case for nonlocal problems.

In the following result, we require $s \neq \frac{1}{2}$ to be able to use the Hardy-type inequality \cite[Theorem 2.3]{mengesha2019fractional}. This is essentially a technicality.

\begin{lemma}[Regularity of control for fractional-type kernels]\label{regularityControlFrac}
Suppose that in the definition of $Z_{ad}$, given in \eqref{linearAdSet}, the functions $a$ and $b$ are constants. Suppose also that, in addition to the contents of Assumption \ref{kernelAssump}, we have that
\begin{equation}\label{kernelOfFracType}
    \frac{k_{\de}(\xi)}{|\xi|^2} \ \sim \ \frac{1}{|\xi|^{n + 2s}} 
\end{equation}
holds for all $\xi \in B_\de(0)$, for some $s \neq \frac{1}{2}$. Then, necessarily, $\overline{g_{\de}} \in H^{s}(\Om; \R^n)$. 
\end{lemma}
\begin{proof} We introduce some notation specifically for this proof. As seen in \cite{mengesha2020fractional}, we denote by $\|u\|_{H^s(\Om; \R^n)}$ the fractional Sobolev norm on vector fields, and denote by $H^s(\Om; \R^n)$ the space of vector fields with finite fractional Sobolev norm.  It has been shown in  \cite[Theorem 1.1]{mengesha2020fractional} that 
the space 
\[
\chi^s(\Om; \R^n) \ := \ \left\{u\in L^2(\Om;\R^n) \ \middle| \ \int_\Om\int_{\Om} {|Du (x, y)|^{2} \over |x-y|^{n+2s}}dydx <\infty \right\}
\] coincides with $H^s(\Om; \R^n)$ with comparable norms.  
Now since  $\overline{p_\de} \in X_0(\Om_{\de};\R^n)$ and $k_\de$ satisfies \eqref{kernelOfFracType}, via direct calculation we have that $\overline{p_\de} \in \chi^s(\Om; \R^n)$, and so it is in $H^{s}(\Om;\R^n)$.   
To finish the proof, we recall the component-wise, pointwise formula
\begin{equation}\label{regularityControlFracEq3}
    \mathbb{P}_{Z_{\text{ad}}}(\overline{p_{\de}}) \ = \ \max\{a, \min\{\overline{p_{\de}}, b\}\},
\end{equation}
proven in \cite[Theorem 2.28]{troltzsch2010optimal}, where we use the assumption that the boxing functions in $Z_{\text{ad}}$ are constants. It is now clear that $   \mathbb{P}_{Z_{\text{ad}}}(\overline{p_{\de}})$ is in $H^{s}(\Om;\R^n)$ from directly estimating the max-min expression. Moreover, $\| \mathbb{P}_{Z_{\text{ad}}}(\overline{p_{\de}})\|_{H^{s}(\Omega;\R^n)} \lesssim  \|\overline{p_{\de}}\|_{H^{s}(\Omega;\R^n)}$. The conclusion for $g_\de$ follows from the formula \eqref{nlCtsCtrlToAdjoint}. 
\end{proof}

\begin{remark}\label{FracControlInterpolationApproach}
An alternative to estimating $\|\overline{g_{\de}}\|_{H^s(\Om; \R^n)}$ directly is to use interpolation theory; see \cite[Chapter 16]{leoni2017first} and \cite[Chapter 25]{tartar2007introduction}. To see this, it suffices to recall that the $H^s(\Omega;\R^n)$ space is an intermediate space between $H^1(\Omega;\R^n)$ and $L^2(\Omega;\R^n)$.
\end{remark}

Having shown that it is possible for the control to lie in a smoother space than $L^2(\Om; \R^n)$, we can proceed with the error analysis. Again, due to the generality of the kernel we are not very explicit in this. Instead, we introduce $\om: \R_+ \to \R_+$ for which $\lim_{h \rightarrow 0^+}\om(h) = 0$. This is such that, if $w \in \P_{Z_{\text{ad}}}X_0(\Om_{\de}; \R^n)$, then
\begin{equation}\label{smoothSeminormFullSpace}
    \|\Pi_0 w - w\|_{L^2(\Om; \R^n)} \ \leq \ \om(h),
\end{equation}
where $\Pi_0: L^2(\Om; \R^n) \rightarrow Z_h$ denotes the $L^2$-projection onto $Z_h$. Clearly, $\om$ depends on the spatial dimension $n$, on the embedding number (or Gelfand width) of the embedding $X_0(\Om_{\de};\R^n) \subset L^2(\Om_{\de};\R^n)$, and on the properties of $\P_{Z_{\text{ad}}}$. In the setting of Lemma \ref{regularityControlFrac} a proper rate of approximation can be established.

\begin{lemma}[Approximation with smoothness]\label{smoothnessapproxseminorm}
Assume that $k_{\de}$ satisfies \eqref{kernelOfFracType} on $B_\de(0)$ for some $s \neq \frac{1}{2}$, then
\begin{equation}\label{smoothSeminormFullSpaceX}
    \|\Pi_0 w - w\|_{L^2(\Om; \R^n)} \ \lesssim \ h^s\|w\|_{X(\Om_{\de}; \R^n)}, \qquad \forall \ w \in X_0(\Om_{\de}; \R^n).
\end{equation}
\end{lemma}

\begin{proof}
The proof repeats that of the Fractional Poincaré Inequality \cite[Lemma 7.1]{ern2017finite} in the vector-valued setting, to obtain the estimate
\begin{equation}\label{smoothSeminormFullSpaceXEq1}
     \|\Pi_0 w - w\|_{L^2(T; \R^n)} \ \lesssim \ h^s[w]_{H^s(T; \R^n)}.
\end{equation}
for each $T \in \Triag_h$. Since our mesh is quasi-uniform, from \eqref{smoothSeminormFullSpaceXEq1} we may deduce, via a localization argument, that
\begin{equation}\label{smoothSeminormFullSpaceXEq2}
\|\Pi_0 w - w\|_{L^2(\Om; \R^n)} \ \lesssim \ h^s[w]_{H^s(\Om; \R^n)}.
\end{equation}
Then invoke \cite[Theorem 1.1]{mengesha2020fractional} to estimate $[w]_{H^s(\Om; \R^n)}$ in terms of the $X(\Om_{\de};\R^n)$ norm.
\end{proof}

We can now obtain an error estimate for the control. In the following result the idea is that, once further regularity of the state/adjoint is known (which can be done for more specific kernels), and a bound on $\om$ like the one in Lemma \ref{smoothnessapproxseminorm} is obtained, the right hand side in the estimate below can  be bounded by a power of $h$.

\begin{theorem}[Convergence of controls]\label{controlerrorestbondp=2}
Assume that $\overline{g_{\de}}$ is the optimal control associated with Problem \ref{nlctsProb}, and $\overline{g_{\de, h}}$ is the discrete optimal control associated with Problem \ref{nlDiscProb}. Then we have the estimate
\begin{equation}\label{errorOfControlsEst}
\|\overline{g_{\de}} - \overline{g_{\de, h}}\|^2_{L^2(\Om; \R^n)} \ \lesssim \ \om(h)^2 + \left(\inf_{v_{\de, h} \in X_{\de, h}}[\overline{u_{\de}} - v_{\de, h}]_{X(\Om_{\de}; \R^n)} \right)^2 + \left(\inf_{v_{\de, h} \in X_{\de, h}}[\overline{p_{\de}} - v_{\de, h}]_{X(\Om_{\de}; \R^n)}\right)^2.
\end{equation}
\end{theorem}
\begin{proof} We follow \emph{a grosso modo} the argument used to prove \cite[Theorem 4.7]{d2019priori}. We let $q_{\de, h} \in X_{\de, h}$ be the Galerkin approximation to $\overline{p_{\de}}$, i.e., the solution of
\begin{equation}\label{controlErrorp=2Ineq5}
    B_{\de}(v_{\de, h}, q_{\de, h}) \ = \ \lang \overline{u_{\de}}, v_{\de, h}\rang \ \quad \ \fa v_{\de, h} \in X_{\de, h}.
\end{equation}
Similarly, $U_{\de, h} \in X_{\de,h}$ is the Galerkin approximation to $\overline{u_{\de}}$:
\begin{equation}\label{controlErrorp=2Ineq10}
    B_{\de}(U_{\de, h}, v_{\de, h}) \ = \ \lang \overline{g_{\de}}, v_{\de, h}\rang \ \quad \ \fa v_{\de, h} \in X_{\de, h}.
\end{equation}
Finally, $r_{\de, h} \in X_{\de, h}$ solves
\begin{equation}\label{controlErrorp=2Ineq7}
    B_{\de}(v_{\de, h}, r_{\de, h}) \ = \ \lang U_{\de, h}, v_{\de, h}\rang \ \quad \ \fa v_{\de, h} \in X_{\de, h}.
\end{equation}
Set $\ga_z := \overline{g_{\de,h}}$ in \eqref{OptCondEq} and $\ga_h := \Pi_0 \overline{g}$ in \eqref{OptCondDEq}. Adding the ensuing inequalities we obtain
\begin{equation}\label{controlErrorp=2Ineq1}
    \la \|\overline{g_{\de}} - \overline{g_{\de, h}}\|^2_{L^2(\Om; \R^n)} \ \leq \ I_1 + I_2,
\end{equation}
where $I_1 := \lang \overline{p_{\de}} - \overline{p_{\de, h}}, \overline{g_{\de, h}} - \overline{g_{\de}}\rang$ and $I_2 := \lang \overline{p_{\de, h}} + \la \overline{g_{\de, h}}, \Pi_0\overline{g_{\de}} - \overline{g_{\de}}\rang$. Now, we write $I_1$ as
\begin{equation}\label{controlErrorp=2Ineq4A}
    \begin{split}I_1 \ &= \ \lang \overline{p_{\de}} - q_{\de, h}, \overline{g_{\de, h}} - \overline{g_{\de}}\rang + 
    \lang q_{\de, h} - r_{\de, h}, \overline{g_{\de, h}} - \overline{g_{\de}}\rang + \lang r_{\de, h} - \overline{p_{\de, h}}, \overline{g_{\de, h}} - \overline{g_{\de}}\rang \\\
    &=: \ I_{1, 1} + I_{1, 2} + I_{1, 3}.
    \end{split}
\end{equation}
We now claim that $I_{1, 3} \leq 0$. Indeed, recall that $\overline{p_{\de, h}}$ denotes the optimal adjoint state for the discrete problem, and hence it satisfies
\begin{equation}\label{controlErrorp=2Ineq8}
B_{\de}(v_{\de, h}, \overline{p_{\de, h}}) \ = \ \lang \overline{u_{\de, h}}, v_{\de, h}\rang \ \quad \ \fa v_{\de, h} \in X_{\de, h}.
\end{equation}
 Subtracting \eqref{controlErrorp=2Ineq7} and \eqref{controlErrorp=2Ineq8} yields
\begin{equation}\label{controlErrorp=2Ineq9}
    B_{\de}(v_{\de, h}, r_{\de, h} - \overline{p_{\de, h}}) \ = \ \lang U_{\de, h} - \overline{u_{\de, h}}, v_{\de, h}\rang \ \quad \ \fa v_{\de, h} \in X_{\de, h},
\end{equation}
while subtracting \eqref{controlErrorp=2Ineq10} and \eqref{stateSystemNLD} shows us that $U_{\de, h} - \overline{u_{\de, h}}$ solves
\begin{equation}\label{controlErrorp=2Ineq12}
    B_{\de}(U_{\de, h} - \overline{u_{\de, h}}, v_{\de, h}) \ = \ \lang \overline{g_{\de}} - \overline{g_{\de, h}}, v_{\de, h}\rang \ \quad \ \fa v_{\de, h} \in X_{\de, h}.
\end{equation}
Set  $v_{\de, h} := r_{\de, h} - \overline{p_{\de, h}}$ in \eqref{controlErrorp=2Ineq12} to obtain
\begin{equation}\label{controlErrorp=2Ineq13}
    B_{\de}(U_{\de, h} - \overline{u_{\de, h}}, r_{\de, h} - \overline{p_{\de, h}}) \ = \ \lang \overline{g_{\de}} - \overline{g_{\de, h}}, r_{\de, h} - \overline{p_{\de, h}}\rang.
\end{equation}
Similarly, in \eqref{controlErrorp=2Ineq9}, we set $v_{\de, h} := U_{\de, h} - \overline{u_{\de, h}}$ and obtain
\begin{equation}\label{controlErrorp=2Ineq14}
B_{\de}(U_{\de, h} - \overline{u_{\de, h}}, r_{\de, h} - \overline{p_{\de, h}}) \ = \ \|U_{\de, h} - \overline{u_{\de, h}}\|^2_{L^2(\Om; \R^n)}.
\end{equation}
Since the left-hand sides of \eqref{controlErrorp=2Ineq13} and \eqref{controlErrorp=2Ineq14} are identical, it follows that $I_{1, 3} \leq 0$. By using Cauchy-Schwarz and C\'ea's lemma repeatedly, we also obtain the following estimates for $I_{1, 1}$ and $I_{1, 2}$:
\begin{equation}\label{controlErrorp=2Ineq15}
  I_{1, 1} \ \lesssim \ \|\overline{g_{\de, h}} - \overline{g_{\de}}\|_{L^2(\Om; \R^n)}\inf_{v_{\de, h} \in X_{\de, h}}[\overline{p_{\de}} - v_{\de, h}]_{X(\Om_{\de}; \R^n)};
\end{equation}
\begin{equation}\label{controlErrorp=2Ineq20}
    I_{1, 2} \ \lesssim \ \|\overline{g_{\de, h}} - \overline{g_{\de}}\|_{L^2(\Om; \R^n)}\inf_{v_{\de, h} \in X_{\de, h}}[\overline{u_{\de}} - v_{\de, h}]_{X(\Om_{\de}; \R^n)}.
\end{equation}
Combine \eqref{controlErrorp=2Ineq15} and \eqref{controlErrorp=2Ineq20}, along with the fact that $I_{1, 3} \leq 0$, to see that
\begin{equation}\label{controlErrorp=2Ineq21}
I_1 \ \lesssim \ \|\overline{g_{\de, h}} - \overline{g_{\de}}\|_{L^2(\Om; \R^n)}\left(\inf_{v_{\de, h} \in X_{\de, h}}[\overline{u_{\de}} - v_{\de, h}]_{X(\Om_{\de}; \R^n)} + \inf_{v_{\de, h} \in X_{\de, h}}[\overline{p_{\de}} - v_{\de, h}]_{X(\Om_{\de}; \R^n)} \right).
\end{equation}
Finally, by two applications of Young's Inequality,
\begin{eqnarray}\label{controlErrorp=2Ineq22}
\begin{aligned}
    I_1 \ \leq \ \frac{\la}{3}\|\overline{g_{\de, h}} - \overline{g_{\de}}\|^2_{L^2(\Om; \R^n)} &+ C\left(\inf_{v_{\de, h} \in X_{\de, h}}[\overline{u_{\de}} - v_{\de, h}]_{X(\Om_{\de}; \R^n)} \right)^2 \\&+ C\left(\inf_{v_{\de, h} \in X_{\de, h}}[\overline{p_{\de}} - v_{\de, h}]_{X(\Om_{\de}; \R^n)}\right)^2,
    \end{aligned}
\end{eqnarray}
for some constant $C > 0$ independent of $\de$ and $h$. Let us now turn our attention to estimating $I_2 = \lang p_h + \la \overline{g_{\de, h}}, \Pi_0 \overline{g_{\de}} - \overline{g_{\de}}\rang$. We write it as
\begin{eqnarray}\label{controlErrorp=2Ineq25}
    \begin{aligned}
        &\lang \overline{p_{\de, h}} + \la \overline{g_{\de, h}}, \Pi_0 \overline{g_{\de}} - \overline{g_{\de}}\rang \ = \ 
        \lang \overline{p_{\de}} + \la \overline{g_{\de}}, \Pi_0 \overline{g_{\de}} - \overline{g_{\de}}\rang + \la\lang \overline{g_{\de, h}} - \overline{g_{\de}}, \Pi_0 \overline{g_{\de}} - \overline{g_{\de}}\rang + \\
        &\lang \overline{p_{\de, h}} - r_{\de, h}, \Pi_0 \overline{g_{\de}} - \overline{g_{\de}}\rang + \lang r_{\de, h} - q_{\de, h}, \Pi_0 \overline{g_{\de}} - \overline{g_{\de}}\rang + \lang q_{\de, h} - \overline{p_{\de}}, \Pi_0 \overline{g_{\de}} - \overline{g_{\de}}\rang \ =: \ \\
        &I_{2, 1} + I_{2, 2} + I_{2, 3} + I_{2, 4} + I_{2, 5},
    \end{aligned}
\end{eqnarray}
and in turn look to control each $I_{2, k}$, $k=1, \ldots,5$. Starting with $I_{2, 1}$, we write it as
\begin{equation}\label{controlErrorp=2Ineq25A}
\lang \overline{p_{\de}} + \la \overline{g_{\de}}, \Pi_0 \overline{g_{\de}} - \overline{g_{\de}}\rang \ = \ \lang \overline{p_{\de}} + \la \overline{g_{\de}} - \Pi_0(\overline{p_{\de}} + \la \overline{g_{\de}}), \Pi_0 \overline{g_{\de}} - \overline{g_{\de}}\rang.
\end{equation}
Since $\overline{g_{\de}} \in \mathbb{P}_{Z_{\text{ad}}}X(\Om; \R^n)$ and $\overline{p_\de}|_{\Om} \in X(\Om;\R^n)$, we use \eqref{smoothSeminormFullSpace} and Cauchy-Schwarz to obtain
\begin{equation}\label{controlErrorp=2Ineq25B}
I_{2, 1} \ \lesssim \ \om(h)^2.
\end{equation}
As for $I_{2, 2}$, we again utilize Cauchy-Schwarz and \eqref{smoothSeminormFullSpace}:
\begin{multline}\label{controlErrorp=2Ineq27}
I_{2, 2} \ = \ \la \lang \overline{g_{\de, h}} - \overline{g_{\de}}, \Pi_0 \overline{g_{\de}} - \overline{g_{\de}}\rang \ \leq \ \la \|\overline{g_{\de, h}} - \overline{g_{\de}}\|_{L^2(\Om; \R^n)}\|\Pi_0 \overline{g_{\de}} - \overline{g_{\de}}\|_{L^2(\Om; \R^n)} \ \leq \\  \la\om(h)\|\overline{g_{\de, h}} - \overline{g_{\de}}\|_{L^2(\Om; \R^n)}\leq \ \frac{\la}{3}\|\overline{g_{\de, h}} - \overline{g_{\de}}\|^2_{L^2(\Om; \R^n)} + \om(h)^2,
\end{multline}
for some constant $C > 0$.

To handle $I_{2, 3}$ we subtract \eqref{controlErrorp=2Ineq7} from \eqref{controlErrorp=2Ineq8}, set $v_{\de, h} := \overline{p_{\de, h}} - r_{\de, h}$ in the result, and obtain
\begin{equation}\label{controlErrorp=2Ineq28}
    [\overline{p_{\de, h}} - r_{\de, h}]_{X(\Om_{\de}; \R^n)} \ \leq \ C\|\overline{u_{\de, h}} - U_{\de, h}\|_{L^2(\Om; \R^n)}.
\end{equation} 
Applying \eqref{smoothSeminormFullSpace} with $w := \overline{p_{\de, h}} - r_{\de, h}$, and combining the result with \eqref{controlErrorp=2Ineq28} gives
\begin{equation}\label{controlErrorp=2Ineq31E}
     \|\overline{p_{\de, h}} - r_{\de, h} - \Pi_0(\overline{p_{\de, h}} - r_{\de, h})\|_{L^2(\Om; \R^n)} \ \leq \ \om(h)\|\overline{u_{\de, h}} - U_{\de, h}\|_{L^2(\Om; \R^n)}
\end{equation} 
\begin{equation}\label{controlErrorp=2Ineq32}
    I_{2, 3} \ \leq \ \om(h)\|\overline{u_{\de, h}} - U_{\de, h}\|_{L^2(\Om; \R^n)},
\end{equation}
Use Young's Inequality and C\'ea's Lemma on \eqref{controlErrorp=2Ineq32} to obtain
\begin{equation}
    I_{2, 3} \ \lesssim \ \om(h)^2 + \left(\inf_{v_{\de, h} \in X_{\de, h}}[\overline{u_{\de}} - v_{\de, h}]_{X(\Om_{\de}; \R^n)}\right)^2.
\end{equation}
To control $I_{2, 4}$, we use Cauchy-Schwarz and \eqref{smoothSeminormFullSpace}:
\begin{equation}\label{controlErrorp=2Ineq33}
    I_{2, 4} \ \leq \ \|r_{\de, h} - q_{\de, h}\|_{L^2(\Om; \R^n)}\|\Pi_0 \overline{g_{\de}} - \overline{g_{\de}}\|_{L^2(\Om; \R^n)} \ \leq \ \om(h)\|r_{\de, h} - q_{\de, h}\|_{L^2(\Om; \R^n)}.
\end{equation}
Then by a standard C\'ea's lemma argument,
\begin{equation}\label{controlErrorp=2Ineq37}
    I_{2, 4} \ \leq \ \om(h)\inf_{v_{\de, h} \in X_{\de, h}}[\overline{u_{\de}} - v_{\de, h}]_{X(\Om_{\de}; \R^n)}.
\end{equation}
Finally, for $I_{2, 5}$ we use \eqref{smoothSeminormFullSpace} and C\'ea's lemma again to obtain
\begin{equation}\label{controlErrorp=2Ineq40}
    I_{2, 5} \ \leq \ \om(h)\inf_{v_{\de, h} \in X_{\de, h}}[\overline{p_{\de}} - v_{\de, h}]_{X(\Om_{\de}; \R^n)}.
\end{equation}
After using Young's Inequality, and combining the estimates for $I_1$ and $I_2$, we conclude
\begin{equation}\label{controlErrorp=2Ineq40A}
    \frac{\la}{3}\|\overline{g_{\de}} - \overline{g_{\de, h}}\|^2_{L^2(\Om; \R^n)} \ \lesssim \ \om(h)^2 + \left(\inf_{v_{\de, h} \in X_{\de, h}}[\overline{u_{\de}} - v_{\de, h}]_{X(\Om_{\de}; \R^n)} \right)^2 + \left(\inf_{v_{\de, h} \in X_{\de, h}}[\overline{p_{\de}} - v_{\de, h}]_{X(\Om_{\de}; \R^n)}\right)^2,
\end{equation}
from which \eqref{errorOfControlsEst} immediately follows.
\end{proof}

We combine and summarize  the state and adjoint error estimates as follows.

\begin{corollary}[Error estimates]\label{fullNormSolnConvp=2}
In the setting of Theorems \ref{FEMErrorEst} and \ref{controlerrorestbondp=2},
\begin{multline}\label{seminormConvergencep=2NoHatUFull} 
   \|\overline{u_{\de}} - \overline{u_{\de, h}}\|_{X(\Om_{\de}; \R^n)} \ \lesssim \ \om(h) + \inf_{v_{\de, h} \in X_{\de, h}}\|\widehat{u_{\de}} - v_{\de, h}\|_{X(\Om_{\de}; \R^n)}\\
    + 
   \inf_{v_{\de, h} \in X_{\de, h}}[\overline{u_{\de}} - v_{\de, h}]_{X(\Om_{\de}; \R^n)} + \inf_{v_{\de, h} \in X_{\de, h}}[\overline{p_{\de}} - v_{\de, h}]_{X(\Om_{\de}; \R^n)};
\end{multline}
\begin{multline}\label{seminormConvergencep=2NoHatPFull} 
     \|\overline{p_{\de}} - \overline{p_{\de, h}}\|_{X(\Om_{\de}; \R^n)} \ \lesssim \ \om(h) + \inf_{v_{\de, h} \in X_{\de, h}}\|\widehat{p_{\de}} - v_{\de, h}\|_{X(\Om_{\de}; \R^n)} + \inf_{v_{\de, h} \in X_{\de, h}}\|\widehat{u_{\de}} - v_{\de, h}\|_{X(\Om_{\de}; \R^n)}\\
      +
     \inf_{v_{\de, h} \in X_{\de, h}}[\overline{u_{\de}} - v_{\de, h}]_{X(\Om_{\de}; \R^n)} + \inf_{v_{\de, h} \in X_{\de, h}}[\overline{p_{\de}} - v_{\de, h}]_{X(\Om_{\de}; \R^n)}.
\end{multline}
\end{corollary}

Furthermore, if the conditions of Lemma \ref{smoothnessapproxseminorm} are satisfied, then $\om(h) \sim h^s$.


\subsection{Error analysis for local problems}\label{refineMeshL}

The analogue of \eqref{OptCondEq} for the local problem is
\begin{eqnarray}\label{OPTCondL}
\begin{aligned}
& \lang \overline{p} + \la \overline{g}, \ga - \overline{g}\rang \ \geq \ 0, \ \quad \fa \ga \in Z_{\text{ad}} \\
  &  \overline{p} \ = \ S^* \overline{u},  \\
    &\overline{u} \ = \ S\overline{g},
\end{aligned}
\end{eqnarray}
where by $S : Z_{\text{ad}} \rightarrow H^1_0(\Om; \R^n)$ we denote the solution operator to problem \ref{locProb}. In a similar manner, the analogue to \eqref{OptCondDEq} for the local discrete problem is
\begin{eqnarray}\label{OPTCondLocalD}
\begin{aligned}
& \lang \overline{p_h} + \la \overline{g_h}, \ga_h - \overline{g_h}\rang \ \geq \ 0, \ \quad \fa \ga_h \in Z_{\text{ad}} \cap Z_h \\
&\overline{p_h} \ = \ S_h^* \overline{u_h} \ = \ S_h \overline{u_h} \\
  &  \overline{u_h} \ = \ S_h\overline{g_h},
\end{aligned}
\end{eqnarray}
where $S_h: Z_h \rightarrow X_h$ denotes the discrete solution operator. The analogue of \eqref{nlCtsCtrlToAdjoint} for the local discrete problem is
\begin{equation}\label{localDiscProjForm}
\overline{g_h}(x) \ = \ -\frac{1}{\la}\P_{Z_{\text{ad}}}(\Pi_0\overline{p_h}(x)). 
\end{equation}
Define the intermediary functions $\widehat{u}, \widehat{p} \in H^1_0(\Om_{\de}; \R^n)$ such that
\begin{equation}\label{discretizeEq3AL}
    B_0(\widehat{u}, v) \ = \ \lang \overline{g_h}, v\rang \ \quad \fa v \in H^1_0(\Om; \R^n); 
\end{equation}
\begin{equation}\label{discretizeEq3BL}
    B_0(v, \widehat{p}) \ = \ \lang v, \widehat{u_h}\rang \ \quad \ \fa v\in H^1_0(\Om; \R^n).
\end{equation}
Again, these functions exist and are uniquely defined thanks to Lax-Milgram. Much like for the non-local problem, we have state and control error estimates as $h \rightarrow 0^+$, and the proofs are virtually identical to those already presented. However, for the local problem, since $\overline{g} \in H^1(\Om; \R^n)$ we may employ the estimate
\begin{equation}\label{alternateSmoothness}
\|\Pi_0 w - w\|_{L^2(\Om; \R^n)} \ \leq \ h[w]_{H^1(\Om; \R^n)}
\end{equation}
in place of \eqref{smoothSeminormFullSpace}. We may also prove, in the same manner as in  Section \ref{refineMeshNL}, error estimates for the discrete state, adjoint, and control. In particular, suppose $(\overline{u}, \overline{g})$ denotes the solution to Problem \ref{lCtsProb}, while $(\overline{u_h}, \overline{g_h})$ denotes the solution to the discrete Problem \ref{lDiscProb}. Assume also that $\overline{p}$ denotes the solution to the adjoint problem \eqref{OPTCondLocalD}, while $\overline{p_h}$ solves the discrete adjoint problem \ref{OPTCondLocalD}. If we further denote $\overline{g}$ as the optimal control for Problem \ref{lCtsProb}, and $\overline{g_h}$ as the discrete optimal control for Problem \ref{lDiscProb}, then we have the estimates
\begin{equation}\label{FEMErrorEstXEq1L}
    \|\overline{u} - \overline{u_h}\|_{H^1(\Om; \R^n)} \ \lesssim \ \inf_{v_h \in X_h}[\widehat{u} - v_h]_{H^1(\Om; \R^n)} + \|\overline{g} - \overline{g_h}\|_{L^2(\Om; \R^n)};
\end{equation}
\begin{equation}\label{FEMErrorEstXEq2L}
    \|\overline{p} - \overline{p_h}\|_{H^1(\Om; \R^n)} \ \lesssim \ \inf_{v_h \in X_h}[\widehat{p} - v_h]_{H^1(\Om; \R^n)} +
\inf_{v_h \in X_h}[\widehat{u} - v_h]_{H^1(\Om; \R^n)} + \|\overline{g} - \overline{g_h}\|_{L^2(\Om; \R^n)}.
\end{equation}
\begin{equation}\label{errorOfControlsEstL}
    \|\overline{g} - \overline{g_h}\|_{L^2(\Om; \R^n)} \ \lesssim \ h + \inf_{v_h \in X_h}[\overline{p} - v_h]_{H^1(\Om; \R^n)} + \inf_{v_h \in X_h}[\overline{u} - v_h]_{H^1(\Om; \R^n)}.
\end{equation}
It then follows that $\overline{u_h} \rightarrow \overline{u}$ and $\overline{p_h} \rightarrow \overline{p}$ in $H^1(\Om; \R^n)$ as $h \rightarrow 0^+$.


\section{Asymptotic Compatibility}\label{compatibility}

In \cite{tian2014asymptotically} the concept of asymptotically compatible schemes for parameter-dependent linear problems was introduced. The goal of asymptotic compatibility is to guarantee that we reach the same local, continuous solution regardless of whether we send $\de$ and $h$ to $0$ separately (in either order) or simultaneously. This broad idea has been implemented extensively in several problems, see  \cite{tian2016asymptotically, buczkowski2022sensitivity, leng2020asymptotically, leng2021asymptotically}. 
Our main goal in this section is to extend this notion to nonlocal optimal control problems, and to show that our ensuing numerical schemes are indeed asymptotically compatible.
We first provide a definition of asymptotic compatibility of a scheme to the optimal control problems that slightly extends \cite[Definition 2.8]{tian2014asymptotically}.  

\begin{definition}[Asymptotic compatibility]\label{asymptoticCompatibilityDef}
We say that the family of solutions $\{(\overline{u_{\de, h}}, \overline{g_{\de, h}} )\}_{h>0,\de > 0}$ to Problem \ref{nlDiscProb} is \textbf{asymptotically compatible} in $\de, h > 0$ if for any sequences $\{\de_k\}^{\infty}_{k = 1}, \{h_k\}^{\infty}_{k = 1}$ with $\de_k, h_k \rightarrow 0$, we have that $\overline{u_{\de_k, h_k}} \rightarrow \overline{u}$ strongly in $L^2(\Om; \R^n)$ and $\overline{g_{\de_k, h_k}} \rightharpoonup \overline{g}$ weakly in $L^2(\Om; \R^n)$. Here $(\overline{u}, \overline{g}) \in H^1_0(\Om; \R^n) \times Z_{\text{ad}}$ denotes
the optimal solution for Problem \eqref{lCtsProb}.
\end{definition}

\tikzcdset{row sep/normal=3cm}
\tikzcdset{column sep/normal=5cm}

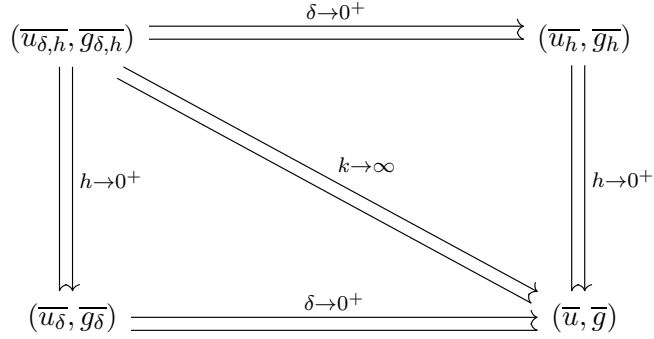
\begin{figure}[ht]
  \begin{tikzcd}
    \arrow[d, "h \rightarrow 0^+"] \arrow[d, rightharpoonup, shift right=1ex]
    \arrow[dr,"k \rightarrow \infty"]
    \arrow[dr,rightharpoonup,shift right =1ex]
        (\overline{u_{\de, h}},\overline{g_{\de, h}})
    \arrow[r, "\de \rightarrow 0^+", shift left=1ex] \arrow[r, rightharpoonup]     
        & (\overline{u_h}, \overline{g_h}) 
        \arrow[d, "h \rightarrow 0^+"] \arrow[d, shift right = 1ex, rightharpoonup] 
    \\
        (\overline{u_{\de}}, \overline{g_{\de}}) 
    \arrow[r, "\de \rightarrow 0^+",] \arrow[r, shift right=1ex, rightharpoonup] 
        & (\overline{u}, \overline{g})
\end{tikzcd}
\caption{Commutative diagram associated with Definition \ref{asymptoticCompatibilityDef}.}
\label{fig:AsympCompat}
\end{figure}

The idea behind asymptotic compatibility can be summarized by saying that the diagram in Figure~\ref{fig:AsympCompat} commutes. The asymptotic compatibility theory for linear problems developed in \cite{tian2014asymptotically} hinges on several structural properties for the operators at hand. Since they will be also useful in our setting, we quickly verify them here as well.

For each $\de > 0$, define $A_{\de}: X_0(\Om_{\de}; \R^n) \rightarrow X_0(\Om_{\de}; \R^n)^*$ as the bounded, invertible, linear operator such that
\begin{equation}\label{AdeBFormBond}
    \lang A_{\de} u, v\rang_{X_0(\Om_{\de}; \R^n)^*, X_0(\Om_{\de}; \R^n)} \ = \ B_{\de}(u, v) \ \quad \fa u, v \in X_0(\Om_{\de}; \R^n).
\end{equation}
Similarly define $A_0: H^1_0(\Om; \R^n) \rightarrow H^{-1}(\Om; \R^n)$ as the bounded, invertible, linear operator such that
\begin{equation}\label{A0BFormBond}
    \lang A_0 u, v\rang_{H^{-1}(\Om; \R^n), H^1_0(\Om; \R^n)} \ = \ B_0(u, v) \ \quad \fa u, v \in H^1_0(\Om; \R^n).
\end{equation}

\begin{proposition}[Asymptotic structural properties]\label{compatibilityCheckBond}
The following hold:
\begin{enumerate}[label=\textbf{AC\arabic*}]
    \item\label{AC1} The family of spaces $\{X_{\de, h}\}_{\de > 0, h > 0}$ is asymptotically dense in $H^1_0(\Om; \R^n)$. That is, given a $v \in H^1_0(\Om; \R^n)$, and some sequences $h_k, \de_k \rightarrow 0$, we can find a sequence $v_k \in X_{\de_k, h_k}$ such that $v_k \rightarrow v$ strongly in $H^1(\Om; \R^n)$ as $k \rightarrow \infty$.
    \item\label{AC2} For any sequences $\{\de_k\}^{\infty}_{k = 1}, \{h_k\}^{\infty}_{k = 1}$ with $\de_k, h_k \rightarrow 0$ and the family of solutions $\{(\overline{u_{\de_k, h_k}}, \overline{g_{\de, h_k}} )\}$ to Problem \ref{nlDiscProb}, there exists a $C > 0$ so that $\|\overline{u_{\de_k, h_k}}\|_{X(\Om_{\de_k}; \R^n)} \leq C$ uniformly in $k \in \N^+$.
    \item\label{AC3} For each $u \in C^{\infty}_0(\Om; \R^n)$ and $\de \geq 0$, we have that $A_{\de}u \in L^2(\Om; \R^n)$.
    \item\label{AC4} For any $u \in C^{\infty}_0(\Om; \R^n)$, we have that $\lim_{\de \rightarrow 0^+}\|A_{\de}u - A_0u\|_{L^2(\Om; \R^n)} \ = \ 0$.
\end{enumerate}
\end{proposition}

\begin{proof}
The fact that finite element spaces of piecewise linear functions are asymptotically dense in $H^1_0(\Om;\R^n)$ is well-known, thus verifying \ref{AC1}. For any $k \in \N$, the bound stated in \ref{AC2} follows from a standard a priori estimate and the fact that $Z_{\text{ad}}$ is bounded in $L^2(\Om;\R^n)$.

Finally, the remaining properties are addressed in \cite{tian2014asymptotically}; see also \cite[Proposition 3.1]{mengesha2014bond}. 
\end{proof}

The structural conditions given above guarantee the asymptotic compatibility for linear problems. Our extension regarding the asymptotic compatibility of our schemes in the setting of optimal control problems is the content of the next result.

\begin{theorem}[Asymptotic compatibility]\label{asymptoticCompatibilityBondLinearp=2}
The solution to Problem 
\ref{nlDiscProb} 
is asymptotically compatible in $\de, h > 0$, in the sense of Definition \ref{asymptoticCompatibilityDef}.
\end{theorem}

\begin{proof}
In this proof we denote $\{(\overline{u_k}, \overline{g_k})\}^{\infty}_{k = 1} := (\overline{u_{\de_k, h_k}}, \overline{g_{\de_k, h_k}})^{\infty}_{k = 1}$, which is the sequence of pairs solving Problem \ref{nlDiscProb}. We also let $\{\overline{p_k}\}^{\infty}_{k = 1} := \{\overline{p_{\de_k, h_k}}\}^{\infty}_{k = 1}$ denote the sequence of solutions to the adjoint problem included in \eqref{OptCondDEq}. We consider an arbitrary, non-relabeled sub-sequence of the triples $\{(\overline{u_k}, \overline{g_k}, \overline{p_k})\}^{\infty}_{k = 1}$, and show that it has a further sub-sequence which always converges to the same limit point. Moreover, this limit solves \eqref{OPTCondL} and, since this uniquely characterizes the solution to Problem \ref{lCtsProb}, asymptotic compatibility will follow.

Since $\{\overline{g_k}\}^{\infty}_{k = 1} \subset Z_{\text{ad}}$, this sequence is bounded in $L^2(\Om; \R^n)$, and there exists a sub-sequence and a function $g_*$ so that $\overline{g_k} \rightharpoonup g_*$ weakly in $L^2(\Om; \R^n)$. Meanwhile,
due to item \ref{AC2} of Proposition \ref{compatibilityCheckBond}, the sequence $\{\overline{u_k}\}_{k=1}^\infty$ is uniformly bounded, and upon taking a further, non-relabeled, sub-sequence, there exists a limit point $u_* \in H^1_0(\Om; \R^n)$ so that $\overline{u_k} \rightarrow u_*$ strongly in $L^2(\Om; \R^n)$. Since $\{(\overline{u_k}, \overline{g_k})\}^{\infty}_{k = 1}$ are pairs satisfying Problem \ref{nlDiscProb}, we have for all $v_k \in X_{\de_k, h_k}$ that
\begin{equation}\label{asymptoticCompatibilityBondLinearp=2Eq1}
    B_{\de_k}(\overline{u_k}, v_k) \ = \ \lang \overline{g_k}, v_k\rang.
\end{equation}
Let $\vphi \in C^{\infty}_0(\Om; \R^n)$ be arbitrary, and denote by $I_k$ the Lagrange nodal interpolant with respect to the mesh of size $h_k$. If $w_k := I_k\vphi \in X_{\de_k, h_k}$, then $w_k \rightarrow \vphi$ in $W^{1, \infty}(\Om; \R^n)$ as $k \rightarrow \infty$. This convergence is sufficiently strong to ensure
\begin{equation}\label{asymptoticCompatibilityBondLinearp=2Eq2}
\lim_{k \rightarrow \infty}\lang \overline{g_k}, w_k\rang \ = \ \lang g_*, \vphi\rang.
\end{equation}
Now, utilizing the definition \eqref{AdeBFormBond}, we write
\begin{equation}\label{asymptoticCompatibilityBondLinearp=2Eq3}
  \begin{split}  B_{\de_k}(\overline{u_k}, w_k) \ &= \ \lang A_{\de_k}\overline{u_k}, w_k\rang_{X_0(\Om_{\de_k}; \R^n)^*, X_0(\Om_{\de_k}; \R^n)} \ \\
    &=  \lang A_{\de_k}\vphi, \overline{u_k}\rang_{X_0(\Om_{\de_k}; \R^n)^*, X_0(\Om_{\de_k}; \R^n)} \\
    &+ \lang A_{\de_k}(w_k - \vphi), \overline{u_k}\rang_{X_0(\Om_{\de_k}; \R^n)^*, X_0(\Om_{\de_k}; \R^n)} \ \\
     &=:\ I_k + II_k.
     \end{split}
\end{equation}
Due to item \ref{AC3} of Proposition \ref{compatibilityCheckBond}, necessarily $A_{\de_k}\vphi \in L^2(\Om; \R^n)$, and by item \ref{AC4}, we have that $A_{\de_k}\vphi \rightarrow A_0\vphi$ strongly in $L^2(\Om; \R^n)$. Due to this and $\overline{u_k} \rightarrow u_*$ strongly in $L^2(\Om; \R^n)$, the term $I$ behaves as follows:
\begin{equation}\label{asymptoticCompatibilityBondLinearp=2Eq4}
    \lim_{k \rightarrow \infty}\lang A_{\de_k}\vphi, \overline{u_k}\rang_{X_0(\Om_{\de_k}; \R^n)^*, X_0(\Om_{\de_k}; \R^n)} \ = \ \lang A_0\vphi, u_*\rang_{H^{-1}(\Om; \R^n), H^1_0(\Om; \R^n)}.
\end{equation}
As for $II_k$, we may use the definition \eqref{A0BFormBond} and that $\overline{u_k}$ is the solution to \eqref{stateSystem}, along with Hölder to deduce
\begin{equation}\label{asymptoticCompatibilityBondLinearp=2Eq5}
    II_k \ = \ B_{\de_k}(\overline{u_k}, w_k - \vphi) \ \lesssim \ \|\overline{u_k}\|_{X(\Om_{\de_k}; \R^n)}\|w_k - \vphi\|_{X(\Om_{\de_k}; \R^n)}.
\end{equation}
Due to item \ref{AC2} the first factor is uniformly bounded in $k$, whereas the second factor is controlled up to a constant (uniform in $k$) by $\|w_k - \vphi\|_{H^1(\Om; \R^n)}$, due to Lemma \ref{ctsXW}. This factor is further bounded from above by $\|w_k - \vphi\|_{W^{1, \infty}(\Om; \R^n)}$, and then the convergence of $w_k \rightarrow \vphi$ in $W^{1, \infty}(\Om; \R^n)$ tells us that $II_k \rightarrow 0$ as $k\to \infty$. The result is that
\begin{equation}\label{asymptoticCompatibilityBondLinearp=2Eq6}
    B_0(u_*, \vphi) \ = \ \lang g_*, \vphi\rang
\end{equation}
for all $\vphi \in C^{\infty}_0(\Om; \R^n)$; by density, we may then extend \eqref{asymptoticCompatibilityBondLinearp=2Eq6} to all $\vphi \in H^1_0(\Om; \R^n)$. 
Repeating the analysis just used for the sequence of states $\{\overline{u_k}\}^{\infty}_{k = 1}$, we identify a $p_* \in H^1_0(\Om; \R^n)$ so that $\overline{p_k} \rightarrow p_*$ strongly in $L^2(\Om; \R^n)$, and
\begin{equation}\label{asymptoticCompatibilityBondLinearp=2Eq7}
    B_0(\vphi, p_*) \ = \ \lang u_*, \vphi\rang 
\end{equation}
for all $\vphi \in H^1_0(\Om; \R^n)$. Now, we link the states, controls and adjoints, beginning as follows: due to \eqref{refineMeshNL}, for each $k$ we have that
\begin{equation}\label{asymptoticCompatibilityBondLinearp=2Eq8}
    B_{\de_k}(\overline{p_k}, v_k) \ = \ \lang \overline{u_k}, v_k\rang
\end{equation}
for all $v_k \in X_{\de_k, h_k}$, and the identity
\begin{equation}\label{asymptoticCompatibilityBondLinearp=2Eq9}
    \overline{g_k}(x) \ = \ -\frac{1}{\la}\P_{Z_{\text{ad}}}(\Pi_0\overline{p_k}(x)).
\end{equation}
The next step is to show that $\Pi_0\overline{p_k} \rightarrow p_*$ strongly in $L^2(\Om; \R^n)$. By the Triangle Inequality and the stability of $\Pi_0$, we estimate
\begin{equation}\label{asymptoticCompatibilityBondLinearp=2Eq9A}
    \|\Pi_0 \overline{p_k} - p_*\|_{L^2(\Om; \R^n)} \ \leq \ \|\overline{p_k} - p_*\|_{L^2(\Om; \R^n)} + \|\Pi_0p_* - p_*\|_{L^2(\Om; \R^n)}.
\end{equation}
Since $\overline{p_k} \rightarrow p_*$ strongly in $L^2(\Om; \R^n)$, the first term in \eqref{asymptoticCompatibilityBondLinearp=2Eq9A} decays to $0$ as $k \rightarrow \infty$, while the second term vanishes due to \eqref{smoothSeminormFullSpace}.

Now, due to the convergence $\Pi_0\overline{p_k} \rightarrow p_*$ strongly in $L^2(\Om; \R^n)$ and the projection mapping being Lipschitz, we have that $-\frac{1}{\la}\P_{Z_{\text{ad}}}(\Pi_0\overline{p_k}(x)) \rightarrow -\frac{1}{\la}\P_{Z_{\text{ad}}}(p_*(x))$; this coupled with the weak convergence $\overline{g_k} \rightharpoonup g_*$ in $L^2(\Om; \R^n)$ lets us conclude
\begin{equation}\label{asymptoticCompatibilityBondLinearp=2Eq10}
    g_*(x) \ = \ -\frac{1}{\la}\P_{Z_{\text{ad}}}(p_*(x)).
\end{equation}

Since \eqref{asymptoticCompatibilityBondLinearp=2Eq6}, \eqref{asymptoticCompatibilityBondLinearp=2Eq7}, and \eqref{asymptoticCompatibilityBondLinearp=2Eq10} all hold, and solutions to the local continuous optimality conditions \eqref{OPTCondL} are necessarily unique, we have that $g_* = \overline{g}$; $p_* = \overline{p}$; and $u_* = \overline{u}$.
Finally, notice that this limit point $(\overline{u}, \overline{g}, \overline{p}) \in H^1_0(\Om; \R^n) \times Z_{\text{ad}} \times H^1_0(\Om; \R^n)$ is independent of the original sub-sequence chosen, which means the entire sequence $\{(\overline{u_k}, \overline{g_k}, \overline{p_k})\}^{\infty}_{k = 1}$ converges to $(\overline{u}, \overline{g}, \overline{p})$ in the $L^2(\Om; \R^n) \times L^2_{\textit{wk}}(\Om; \R^n) \times L^2(\Om; \R^n)$ topology, completing the proof.
\end{proof}


\section*{Acknowledgements}
TM is supported by NSF grants DMS-1910180 and DMS-2206252.
AJS and JMS have been supported by NSF grant DMS-2111228.


\bibliographystyle{plain} 
\bibliography{refs} 

\end{document}